\documentclass{amsart}
\usepackage{amssymb}
\usepackage{verbatim}
\usepackage{array}
\usepackage{latexsym}
\usepackage{enumerate}

\usepackage{amsmath,mathtools}
\usepackage{mathrsfs}
\usepackage{amsfonts}
\usepackage{amsthm}
\usepackage[foot]{amsaddr}
\usepackage[table,svgnames,usenames,dvipsnames]{xcolor}
\usepackage{caption}
\usepackage[english]{babel}
\usepackage{graphicx}
\usepackage{tikz}
\usepackage{tikz-cd}
\usepackage{adjustbox}
\usepackage{pgfplots}
\pgfplotsset{compat=newest}
\usepackage{nicematrix}

\usepackage{csquotes}
\usepackage{hyperref}
\usepackage[capitalise]{cleveref}

\newtheorem{theorem}{Theorem}[section]
\newtheorem{proposition}[theorem]{Proposition}
\newtheorem{lemma}[theorem]{Lemma}
\newtheorem{corollary}[theorem]{Corollary}
\newtheorem{definition}[theorem]{Definition}

\newtheorem{example}[theorem]{Example}

\newtheorem{remark}[theorem]{Remark}

\def\A{\mathbb A}
\newcommand{\Fq}{{\mathbb{F}_q}}
\newcommand{\Fqm}{{\mathbb{F}_{q^m}}}
\newcommand{\F}{\mathbb{F}}

\newcommand{\E}{\mathcal{E}}
\newcommand{\M}{\mathcal{M}}

\newcommand{\PP}{\mathcal{P}}
\newcommand{\CC}{\mathfrak{C}}
\newcommand{\EE}{\bar{\E}}
\newcommand{\MM}{\bar{\M}}

\renewcommand{\a}{\mathfrak{a}}
\renewcommand{\P}{\mathbb{P}}
\renewcommand{\t}{\bar{t}}

\title{Locally Recoverable Codes with availability from a family of fibered surfaces}
\date{}

\author[Cecília Salgado]{Cecília Salgado$^1$} \address{$^1$Faculty of Science and Engineering - Bernoulli Institute, University of Groningen, Groningen 9747AG, The Netherlands} \email{c.salgado@rug.nl} \thanks{}
\author[Lara Vicino]{Lara Vicino$^{1,2}$} \address{}  \email{l.vicino@rug.nl} \thanks{}
\address{$^2$Department of Applied Mathematics and Computer Science, Technical University of Denmark, Kongens Lyngby 2800, Denmark}

\begin{document}

\begin{abstract}
    We construct Locally Recoverable Codes (LRCs) with availability $2$ from a family of fibered surfaces. To obtain the locality and availability properties, and to estimate the minimum distance of the codes, we combine techniques coming from the theory of one-variable function fields and from the theory of fibrations on surfaces. When the  locality parameter is $r=3$, we obtain a sharp bound on the minimum distance of the codes. In that case, we give a geometric interpretation of our codes in terms of doubly elliptic surfaces. In particular, this provides the first instance of an error correcting code constructed using a (doubly elliptic) K3 surface.
\end{abstract}

\maketitle

\thanks{{\em Keywords}: Locally Recoverable Codes, availability, AG codes, one-variable function fields, fibered surfaces.

\thanks{{\em Subject classifications}: 14H05, 14G50, 94B27.

\section{Introduction}
\label{sec:intro}

\emph{Locally Recoverable Codes} (LRCs for short) are a type of error-correcting codes particularly suited for applications in distributed storage systems, where data is split and stored across multiple servers. Indeed, one of the most common issues in such systems is the scenario of a server failure, which entails that portions of data become suddenly unavailable. However, if data is encoded by means of an LRC, then the retrieval of missing information is possible by accessing only a small portion of the remaining data, which makes the recovery procedure fast and cost-effective. For this reason, LRCs have been extensively studied over the last two decades, in particular those with additional properties like \emph{availability} or \emph{hierarchical locality}, which add an extra layer of data protection by ensuring that there are multiple subsets of information symbols that can be used to retrieve a missing symbol. 

Initially studied in \cite{HLM, DGWWR, GHSY, HCL, TB14b}, LRCs have since been obtained with several different techniques, including constructions as evaluation codes from algebraic varieties over finite fields. These constructions as Algebraic Geometry (AG) codes allow us to interpret and deduce properties of the obtained LRCs by studying geometric properties of the underlying varieties. Several examples of constructions of LRCs as AG codes can be found in the literature, for instance from covering maps of curves \cite{BBV21}, fiber products of curves \cite{HMM18, HMM25}, elliptic and Hirzebruch surfaces \cite{SVAV21}, Artin-Schreier surfaces \cite{BMW24} and projective bundles \cite{AAAOAVA24}. 

In this paper, we construct locally recoverable codes with availability from a particular family of fibered surfaces. Our construction is partially inspired by \cite{AAAOAVA24}, where a family of asymptotically good LRCs with availability was obtained from projective space bundles. Indeed, the idea underlying our approach is to consider a surface admitting two distinct fibrations to a projective line and to construct LRCs by evaluating a suitably chosen space of functions on such surface at points that lie on a curve contained in the surface that is horizonal to both fibrations, i.e., intersects all fibers. We choose the curve in such a way that each evaluation point has two disjoint recovery sets, each contained in the fiber of one of the two distinct fibrations that contains the said point. Since two distinct fibers (of the same fibration) do not intersect, the recovery sets obtained in this way are naturally disjoint.

Note that, in this way, the codes we obtain can also be directly interpreted as AG codes from a curve, which makes it possible to use the theory of one-variable function fields and Newton arcs of curves to study their properties. However, since the ideas for obtaining locality and availability in our construction were inspired by viewing our codes as evaluation codes on a surface, in the spirit of \cite{SVAV21} and \cite{AAAOAVA24}, we choose to also adopt the framework of working on the surface, with explicit coordinates, in $\P^2 \times \P^1$. A further reason for this is to shed light on an interpretation of our construction, described in \cref{sec:genus1:fibrations}, where we consider a particular surface in the family that we describe and outline a different recovery procedure, based on the surface having two genus $1$ fibrations. This can be seen as a natural generalization of the results contained in \cite[Section VI]{SVAV21}.

For certain alphabets and certain values of the locality parameter $r$, the codes that we construct can be compared to those obtained in \cite[Section IV]{TB14}, although the comparison is not straightforward. Indeed, for a fixed alphabet, it is not a priori clear whether our construction and the construction in \cite[Section IV]{TB14} can produce a code of the same length and with the same locality parameter $r$.
Moreover, while the construction in \cite[Section IV]{TB14} is solely based on the additive and multiplicative structure of the field, our construction builds upon the study of the splitting of certain places in some function field extensions. While our construction has the downside of being less flexible and of requiring, in general, larger alphabet sizes, it has the advantage of making it possible to obtain a lower bound for the minimum distance which is sharp for $r=3$. The case $r=3$ is the most directly relevant for applications, since having a low number of symbols in a recovery set leads to more efficient recovery procedures (see for instance \cite{GHSY}).

The paper is structured as follows: in \cref{sec:background}, we recall several background notions that are used throughout the paper. We start with definitions on locally recoverable codes and AG codes, and we then recall some results from the theory of one-variable function fields and Newton arcs. We conclude the section by including some essential results on fibered surfaces. 
In \cref{sec:code:construction}, we outline our construction of LRCs from a family of fibered surfaces. In \cref{sec:parameters}, we study the parameters (dimension, locality, availability, minimum distance) of the codes obtained in the previous section, presenting explicit examples in \cref{table:codes}. Finally, in \cref{sec:genus1:fibrations}, we outline a reinterpretation of the construction from one of the surfaces in our family, showing how, in this case, a different recovery procedure can be devised by using specific properties of the surface, which in this case is a K3 surface with two genus 1 fibrations.

\section{Background}
\label{sec:background}

In this section, we introduce the notations and background theory that are used throughout the paper. We start by recalling some definitions on Locally Recoverable Codes and Algebraic Geometry codes, then we proceed to recall some results on Newton arcs of algebraic curves. We conclude the section with a collection of definitions and facts on fibrations and fibered surfaces. 

\subsection{Locally Recoverable Codes}

Let $q=p^s$, with $p$ a prime number and $s\in \mathbb{Z}_{\geq 1}$, and let $\Fq$ be the finite field with $q$ elements. A \emph{linear} $[n,k,d]_q$ \emph{error-correcting code} $C$ over the alphabet $\Fq$ is a linear $\Fq$-subspace of $\Fq^n$. In what follows, we simply refer to such an object as a \emph{code}. The vectors in $C$ are called \emph{codewords}.
The positive integer $n$ is the \emph{length} of the code $C$, while $k$ is the \emph{dimension} of $C$, which is its dimension as an $\Fq$-vector space. The integer $d$ is called the \emph{minimum distance} of $C$, and it is the parameter measuring how many errors or erasures the code can detect and correct: the larger the minimum distance, the more errors or erasures can be detected and corrected. The parameters $n, k, d$ of a code $C$ are related by the celebrated \emph{Singleton bound} $d \leq n - k + 1$, which gives a tradeoff between the distance and the dimension of $C$.

In this paper, we study a particular type of codes having the properties of \emph{locality} and \emph{availability}.

\begin{definition}
\label[definition]{def:LRC}
      A code $C\subset \mathbb{F}_{q}^n$ is a \emph{Locally Recoverable Code} (LRC) with locality $r$ if and only if for all codewords $\mathfrak{c}=(\mathfrak{c}_1,\ldots, \mathfrak{c}_n)\in C$ and for all $i \in \{1,\cdots, n\}$, there exists a subset $I_i\subseteq \{1, \cdots, n\} \setminus \{i\}$ of cardinality $|I_i|=r$ and a function $\varphi_i: \Fq^r \longrightarrow \Fq$ such that $\mathfrak{c}_i = \varphi_i(\mathfrak{c}|_{I_i})$.

      We say that an LRC $C$ has \emph{availability} $\mathfrak{a}$ with locality $(r_1,\ldots,r_{\mathfrak{a}})$ if and only if for all $i \in \{1,\cdots, n\}$ there exist $I_{i1}, \dots, I_{i\mathfrak{a}}\subseteq \{1, \cdots, n\} \setminus \{i\}$ with $|I_{ij}|=r_j$, for all $j=1,\ldots, \mathfrak{a}$, and $I_{ij} \cap I_{ik} = \emptyset$ for any $j \neq k$, and for all $\mathfrak{c}\in C$, $\mathfrak{c}_i = \varphi_{ij}(\mathfrak{c}|_{I_{ij}})$ for some function $\varphi_{ij}:\Fq^{r_j}\longrightarrow \Fq$.  
\end{definition}
Given an LRC with no availability, losing a symbol of a recovery set compromises the locality property of the code. For this reason, the property of availability of an LRC can be seen as adding an extra layer of data protection, since there are multiple disjoint recovery sets for each symbol of every codeword. 

In the paper, we construct and study the parameters of a certain class of LRCs with availability $\mathfrak{a}=2$. More precisely, we obtain these codes as evaluation codes constructed from a family of algebraic surfaces over $\Fq$. In the following subsection, we  introduce some essential background notions on Algebraic Geometry (AG) codes.

\subsection{Algebraic Geometry codes}

\emph{Algebraic Geometry (AG) codes} are evaluation codes obtained from algebraic varieties over finite fields. They were first introduced by V.D.~Goppa between the 70's and the 80's, see \cite{G70, G71, G77, G81, G82}, with a construction from curves over finite fields. In this paper, we construct AG codes from certain surfaces, according to the following \cref{def:AGcodes}, which generalises the one introduced by V.D.~Goppa for curves. 

\begin{definition}
\label[definition]{def:AGcodes}
    Let $\mathcal{X}$ be a quasi-projective variety over a finite field $\Fq$ and denote by $\mathcal{X}(\Fq)$ the set of $\Fq$-rational points of $\mathcal{X}$ and by $\Fq(\mathcal{X})$ the ($\Fq$-rational) function field of $\mathcal{X}$. Let $S:=\{P_1,\ldots,P_n\}\subseteq \mathcal{X}(\Fq)$ be a subset of cardinality $|S| = n$ and $V$ be a finite-dimensional subspace of $\Fq(\mathcal{X})$ of rational functions with no poles at any of the points in $S$. Define the linear evaluation map
        \begin{align*}
        \mathrm{ev}_{S}: \ &V \longrightarrow \Fq^n\\
        &f \longmapsto \left(f(P_1), \ldots, f(P_n)\right). 
    \end{align*}
    The \emph{Algebraic Geometry (AG) code} $C(\mathcal{X},S,V)$ is defined as $C(\mathcal{X},S,V):=\mathrm{im} \ \mathrm{ev}_{S}$, that is, 
    \begin{equation*}
        C(\mathcal{X},S,V) = \{\left(f(P_1), \ldots, f(P_n)\right) \mid f\in V\}.
    \end{equation*}
\end{definition}

A code $C(\mathcal{X},S,V)$ as in \cref{def:AGcodes} has length $n = |S|$ and dimension $k = \mathrm{dim}_{\Fq} \ V - \mathrm{dim}_\Fq \ \mathrm{ker} \ \mathrm{ev}_{S}$. In particular, if the evaluation map $\mathrm{ev}_{S}$ is injective, then the dimension of $C(\mathcal{X},S,V)$ is simply $k = \mathrm{dim}_{\Fq} \ V$. 

Since AG codes are constructed by evaluating functions on an algebraic variety, they have the peculiar feature that their parameters and properties can be understood and studied in terms of properties of the underlying variety. For instance, bounds on the value of the minimum distance $d$ of an AG code $C(\mathcal{X},S,V)$ can be derived by studying the vanishing of the functions in $V$ at the points in $S$. If one can show that any function in $V$ has at most $N$ distinct zeros in the set $S$, then the lower bound $d \geq n - N$ follows. In \cref{subsec:d:bound}, we use such a strategy to prove a lower bound on the minimum distance of the codes constructed in \cref{sec:code:construction}.

Moreover, when constructing LRCs as AG codes, certain features of the underlying variety can be instrumental to devise recovery procedures and obtain further properties such as, for instance, the availability property (see \cref{sec:code:construction} and \cref{sec:parameters}).

\subsection{Newton arcs}
\label{subsec:newton:background}

Let $F$ be a one-variable function field over $\Fq$, i.e., the function field of some curve $\mathcal{X}$ over $\Fq$. 
In this subsection, we recall some fundamental definitions and results on the \emph{Newton arc} associated to a polynomial in $F[T]$ with respect to some place of $F$. An extensive exposition of the theory and results summarised in this subsection can be found in the (unpublished) notes by P.~Beelen \cite{Bnotes}. On the other hand, a corresponding exposition in the setting of number fields can be found in \cite{EF21}. Further references on Newton arcs and polygons of curves are \cite{BP00, Bthesis}. Regarding the theory and notations on one-variable function fields, we instead refer to \cite{Sti}, where Kummer's theorem can be found in \cite[Theorem 3.3.7]{Sti}.

\begin{definition}[\cite{Bnotes}]
\label[definition]{def:supportset}
    Let $Q$ be a place of $F$ and $\varphi(T):=\sum_{i=0}^{m}a_iT^i\in F[T]$ a polynomial of degree $m\geq 0$. The \emph{support set} of $\varphi(T)$ at $Q$ is defined as
    \begin{equation*}
        \mathcal{S}_{\varphi(T),Q}:=\{(i,v_Q(a_i)) \mid 0\leq i \leq m, \ a_i\neq 0\},
    \end{equation*}
    where $v_Q(a_i)$ denotes the valuation of the function $a_i$ at the place $Q$. When the place to which the support set refers to is clear, we simply write $\mathcal{S}_{\varphi(T)}$. 
\end{definition}

\begin{definition}[\cite{Bnotes}]
\label[definition]{def:newtonarc}
    Let $Q$ be a place of $F$, $\varphi(T):=\sum_{i=0}^{m}a_iT^i\in F[T]$ a polynomial of degree $m\geq 0$ and $\mathcal{S}_{\varphi(T),Q}$ as in \cref{def:supportset}.
    For any $\rho,c\in \mathbb{Q}$, consider the half plane $H_{\rho,c}$ defined by
    \begin{equation*}
        H_{\rho,c}:=\{(i,j)\in \mathbb{R}^2 \mid \rho i + j \geq c\}
    \end{equation*}
    and observe that, for any $\rho\in \mathbb{Q}$, there exists a unique $\tilde{c}$ such that $\mathcal{S}_{\varphi(T),Q}\subseteq H_{\rho,\tilde{c}}$ and there exists a point in $\mathcal{S}_{\varphi(T),Q}$ with $\rho i + v_Q(a_i) l = \tilde{c}$. For a given value of $\rho\in \mathbb{Q}$, we call the half plane $H_{\rho,\tilde{c}}$ the \emph{best-fitting} half plane with respect to $\rho$ and we define the lower convex hull of $\mathcal{S}_{\varphi(T),Q}$ as $\Delta_Q(\varphi(T)):=\cap_{\rho\in \mathbb{Q}}H_{\rho,\tilde{c}}$. 

    Then, the boundary of $\Delta_Q(\varphi(T))$ consists of a finite number of line segments, namely: two vertical ones defined by $i=m$ and $i=a$, where $a\in \mathbb{Z}_{\geq 0}$ is the smallest possible integer such that $T^a \mid \varphi(T)$, and finitely many other non-vertical segments.

    The \emph{Newton arc} $\Gamma_Q(\varphi(T))$ of $\varphi(T)$ with respect to $Q$ is defined to be the union of the non-vertical segments in the boundary of $\Delta_Q(\varphi(T))$.  
\end{definition}

Let $\mathcal{L}_1,\ldots, \mathcal{L}_h$, $h\in \mathbb{Z}_{\geq 1}$, be the distinct non-vertical line segments in the boundary of $\Delta_Q(\varphi(T))$, and denote by $\rho(\mathcal{L}_i)$ the slope of the segment $\mathcal{L}_i$. The Newton arc $\Gamma_Q(\varphi(T))$ of $\varphi(T)$ with respect to $Q$ can hence be uniquely described as $\Gamma_Q(\varphi(T))=\cup_{i=1}^{h} \mathcal{L}_i$, where the segments $\mathcal{L}_1,\ldots, \mathcal{L}_h$ are listed in order of strictly increasing slope, i.e., $\rho(\mathcal{L}_1) < \cdots < \rho(\mathcal{L}_h)$.

\begin{center}
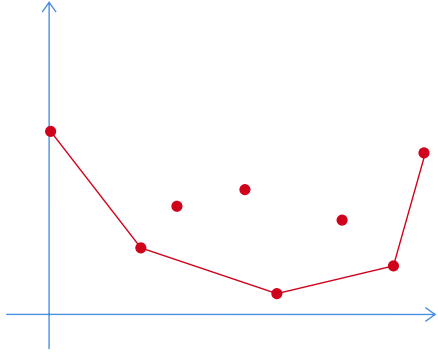
\begin{figure}[ht]
\adjustbox{scale=0.7}{%

\tikzset{every picture/.style={line width=0.75pt}} 

\begin{tikzpicture}[x=0.75pt,y=0.75pt,yscale=-1,xscale=1]

\draw [color={rgb, 255:red, 74; green, 144; blue, 226 }  ,draw opacity=1 ][line width=0.75]  (167.98,308.97) -- (476.98,309.03)(198.92,83.98) -- (198.88,333.98) (469.98,304.03) -- (476.98,309.03) -- (469.98,314.03) (193.92,90.98) -- (198.92,83.98) -- (203.92,90.98)  ;
\draw [color={rgb, 255:red, 208; green, 2; blue, 27 }  ,draw opacity=1 ]   (200,177) -- (265,261) ;
\draw [color={rgb, 255:red, 208; green, 2; blue, 27 }  ,draw opacity=1 ]   (265,261) -- (363,294) ;
\draw [color={rgb, 255:red, 208; green, 2; blue, 27 }  ,draw opacity=1 ]   (363,294) -- (447,274) ;
\draw [color={rgb, 255:red, 208; green, 2; blue, 27 }  ,draw opacity=1 ]   (447,274) -- (469,196) ;
\draw  [color={rgb, 255:red, 208; green, 2; blue, 27 }  ,draw opacity=1 ][fill={rgb, 255:red, 208; green, 2; blue, 27 }  ,fill opacity=1 ] (196.5,177) .. controls (196.5,175.07) and (198.07,173.5) .. (200,173.5) .. controls (201.93,173.5) and (203.5,175.07) .. (203.5,177) .. controls (203.5,178.93) and (201.93,180.5) .. (200,180.5) .. controls (198.07,180.5) and (196.5,178.93) .. (196.5,177) -- cycle ;
\draw  [color={rgb, 255:red, 208; green, 2; blue, 27 }  ,draw opacity=1 ][fill={rgb, 255:red, 208; green, 2; blue, 27 }  ,fill opacity=1 ] (261.5,261) .. controls (261.5,259.07) and (263.07,257.5) .. (265,257.5) .. controls (266.93,257.5) and (268.5,259.07) .. (268.5,261) .. controls (268.5,262.93) and (266.93,264.5) .. (265,264.5) .. controls (263.07,264.5) and (261.5,262.93) .. (261.5,261) -- cycle ;
\draw  [color={rgb, 255:red, 208; green, 2; blue, 27 }  ,draw opacity=1 ][fill={rgb, 255:red, 208; green, 2; blue, 27 }  ,fill opacity=1 ] (359.5,294) .. controls (359.5,292.07) and (361.07,290.5) .. (363,290.5) .. controls (364.93,290.5) and (366.5,292.07) .. (366.5,294) .. controls (366.5,295.93) and (364.93,297.5) .. (363,297.5) .. controls (361.07,297.5) and (359.5,295.93) .. (359.5,294) -- cycle ;
\draw  [color={rgb, 255:red, 208; green, 2; blue, 27 }  ,draw opacity=1 ][fill={rgb, 255:red, 208; green, 2; blue, 27 }  ,fill opacity=1 ] (443.5,274) .. controls (443.5,272.07) and (445.07,270.5) .. (447,270.5) .. controls (448.93,270.5) and (450.5,272.07) .. (450.5,274) .. controls (450.5,275.93) and (448.93,277.5) .. (447,277.5) .. controls (445.07,277.5) and (443.5,275.93) .. (443.5,274) -- cycle ;
\draw  [color={rgb, 255:red, 208; green, 2; blue, 27 }  ,draw opacity=1 ][fill={rgb, 255:red, 208; green, 2; blue, 27 }  ,fill opacity=1 ] (465.5,192.5) .. controls (465.5,190.57) and (467.07,189) .. (469,189) .. controls (470.93,189) and (472.5,190.57) .. (472.5,192.5) .. controls (472.5,194.43) and (470.93,196) .. (469,196) .. controls (467.07,196) and (465.5,194.43) .. (465.5,192.5) -- cycle ;
\draw  [color={rgb, 255:red, 208; green, 2; blue, 27 }  ,draw opacity=1 ][fill={rgb, 255:red, 208; green, 2; blue, 27 }  ,fill opacity=1 ] (287.5,231) .. controls (287.5,229.07) and (289.07,227.5) .. (291,227.5) .. controls (292.93,227.5) and (294.5,229.07) .. (294.5,231) .. controls (294.5,232.93) and (292.93,234.5) .. (291,234.5) .. controls (289.07,234.5) and (287.5,232.93) .. (287.5,231) -- cycle ;
\draw  [color={rgb, 255:red, 208; green, 2; blue, 27 }  ,draw opacity=1 ][fill={rgb, 255:red, 208; green, 2; blue, 27 }  ,fill opacity=1 ] (336.5,219) .. controls (336.5,217.07) and (338.07,215.5) .. (340,215.5) .. controls (341.93,215.5) and (343.5,217.07) .. (343.5,219) .. controls (343.5,220.93) and (341.93,222.5) .. (340,222.5) .. controls (338.07,222.5) and (336.5,220.93) .. (336.5,219) -- cycle ;
\draw  [color={rgb, 255:red, 208; green, 2; blue, 27 }  ,draw opacity=1 ][fill={rgb, 255:red, 208; green, 2; blue, 27 }  ,fill opacity=1 ] (406.5,241) .. controls (406.5,239.07) and (408.07,237.5) .. (410,237.5) .. controls (411.93,237.5) and (413.5,239.07) .. (413.5,241) .. controls (413.5,242.93) and (411.93,244.5) .. (410,244.5) .. controls (408.07,244.5) and (406.5,242.93) .. (406.5,241) -- cycle ;

\draw (239,269.4) node [anchor=north west][inner sep=0.75pt]  [font=\scriptsize,color={rgb, 255:red, 74; green, 144; blue, 226 }  ,opacity=1 ]  {$( i,v_{Q}( a_{i}))$};

\end{tikzpicture}
}
\caption{General shape of the Newton arc of a polynomial $\varphi(T)$ with respect to a place $Q$ of $F$. \label{fig:newtonarc:gen}}
\end{figure}
\end{center}

Let now $\mathcal{O}_Q$ be the DVR associated to the place $Q$. Then, given any $\varphi(T)\in F[T]$, there always exists some element $f\in F\setminus\{0\}$ such that $f\varphi(T)\in \mathcal{O}_Q$[T] and the valuation of at least one coefficient of $f\varphi(T)$ at $Q$ is zero. The Newton arc $\Gamma_Q(f\varphi(T))$ with respect to $Q$ of the polynomial $f\varphi(T)\in F[T]$ obtained in this way is then just a vertical translate of $\Gamma_Q(\varphi(T))$. In particular, the slopes of the segments of $\Gamma_Q(\varphi(T))$ are the same as those of the segments of $\Gamma_Q(f\varphi(T))$. Hence, to study the Newton arc of $\varphi(T)$ with respect to a place $Q$, we can always study instead the Newton arc of $f\varphi(T)\in \mathcal{O}_Q[T]$ with respect to $Q$. 

We henceforth assume, without loss of generality, that $\varphi(T)\in \mathcal{O}_Q[T]$. Note that we can rewrite $\varphi(T)$ as 
\begin{equation*}
    \varphi(T) = \sum_{i=0}^m c_i t_Q^{\alpha_i}T^i,
\end{equation*}
where $c_i\in \mathcal{O}_Q^*$ for all $i=0,\ldots, m$, $t_Q$ is a local parameter at $Q$ and $\alpha_i\in \mathbb{Z}_{\geq 0}$. Let $\mathcal{L}$ be the line segment of $\Gamma_Q(\varphi(T))$ with starting point $(i, \alpha_i)$ and end point $(j,\alpha_j)$. We define the following monic polynomial associated to $\mathcal{L}$
\begin{equation}
\label{eq:gammaL:def}
    \gamma_\mathcal{L}(T):=\bar{c}_j^{-1}\sum_{\ell \ : \ (\ell,\alpha_\ell)\in \mathcal{L}} \bar{c}_\ell T^{\ell-i} \in (\mathcal{O}_Q/Q)[T],
\end{equation}
where $\bar{c}_\ell$ is the the reduction of $c_\ell$ modulo $Q$, for all $\ell=0,\ldots, m$.

While the polynomial $\gamma_\mathcal{L}(T)$ depends on the choice of the local parameter $t_Q$, its degree $(j-i)$ does not depend on this choice. Moreover, with the notations introduced before, let $\rho(\mathcal{L})$ denote the slope of the segment $\mathcal{L}$. Then we can write uniquely $\rho(\mathcal{L}) = -a/b$, where $a,b\in \mathbb{Z}$, $b>0$ and $\mathrm{gcd} \ (a,b) = 1$. If $\mathcal{L}$ has starting point $(i, \alpha_i)$ and end point $(j,\alpha_j)$, we then have that 
\begin{equation*}
    -\frac{a}{b} = \frac{\alpha_\ell - \alpha_i}{\ell-i} 
\end{equation*}
for all $\ell$ such that $(\ell, \alpha_\ell)\in \mathcal{L}$, and since $\alpha_\ell - \alpha_i\in \mathbb{Z}$ and $\mathrm{gcd} \ (a,b) = 1$, we conclude that $b \mid (\ell-i)$ for all $\ell$. This means that $\gamma_\mathcal{L}(T)$ can be written as a polynomial in $T^b$, that is,
\begin{equation}
\label{eq:deltaL:def}
    \gamma_\mathcal{L}(T) = \delta_\mathcal{L}(T^b),
\end{equation}
where the polynomial $\delta_\mathcal{L}(T) \in (\mathcal{O}_Q/Q)[T]$ is defined by \cref{eq:deltaL:def}.

Like the polynomial $\gamma_\mathcal{L}(T)$, also the polynomial $\delta_\mathcal{L}(T)$ depends on the choice of the local parameter $t_Q$. However, it can be shown that the factorization pattern of $\delta_\mathcal{L}(T)$ is actually not affected by the choice of the local parameter, i.e., that the degree, separability and multiplicity of the irreducible factors of $\delta_\mathcal{L}(T)$ do not change, disregarding the choice of the local parameter at $Q$. This is of importance for the following result, which we use in \cref{subsec:d:bound}.

\begin{theorem}[\cite{Bnotes}]
    \label{thm:newtonarc:splitting}
    Let $F':=F(s)$ be an extension of $F$ of degree $[F':F]=m$. Furthermore, let $Q$ be a place of $F$ and $t_Q$ a local parameter at $Q$. Let $\varphi(T)\in F[T]$ be a polynomial of degree $m$ such that $\varphi(s)=0$. 

    Let $\mathcal{L}$ be a line segment of the Newton arc $\Gamma_Q(\varphi(T))$ of $\varphi(T)$ with respect to $Q$, with slope $\rho(\mathcal{L})= -a/b$, with $a,b\in \mathbb{Z}$, $b>0$ and $\mathrm{gcd} \ (a,b) = 1$. Let $\delta_{\mathcal{L},i}(T) \in (\mathcal{O}_Q/Q)[T]$ be an irreducible monic factor of $\delta_\mathcal{L}(T)$ and let $\psi_{\mathcal{L},i}(T)\in \mathcal{O}[T]$ be a monic polynomial such that
    \begin{equation*}
        \bar{\psi}_{\mathcal{L},i}(T) = \delta_{\mathcal{L},i}(T) \quad \mbox{and} \quad \mathrm{deg} \ \psi_{\mathcal{L},i}(T) = \mathrm{deg} \ \delta_{\mathcal{L},i}(T),
    \end{equation*}
    where $\bar{\psi}_{\mathcal{L},i}(T)$ denotes the polynomial in $(\mathcal{O}_Q/Q)[T]$ whose coefficients are the reduction modulo $Q$ of the coefficients of $\psi_{\mathcal{L},i}(T)$.

    Then there exists a place $P_{\mathcal{L},i}$ of $F'$ lying over $Q$ such that $bv_{P_{\mathcal{L},i}}(s) = a e(P_{\mathcal{L},i}\mid Q)$ and $\psi_{\mathcal{L},i}(t_Q^{-a}s^b)\in P_{\mathcal{L},i}$. Here $e(P_{\mathcal{L},i}\mid Q)$ denotes the ramification index of $P_{\mathcal{L},i}$ over $Q$.
    Moreover, all places $P_{\mathcal{L},i}$ obtained in this way are distinct.
\end{theorem}

\cref{thm:newtonarc:splitting} has, as a consequence, the following corollary, which we make use of in the discussion in \cref{subsec:d:bound}.

\begin{corollary}[\cite{Bnotes}]
    \label{cor:newtonarc:splitting}
    Let $F':=F(s)$ be an extension of $F$ of degree $[F':F]=m$. Furthermore, let $Q$ be a place of $F$ and $t_Q$ a local parameter at $Q$. Let $\varphi(T)\in F[T]$ be a polynomial of degree $m$ such that $\varphi(s)=0$. 

    If for each edge $\mathcal{L}$ of the Newton arc $\Gamma_Q(\varphi(T))$ the polynomial $\delta_\mathcal{L}(T)$ is squarefree as an element of $(\mathcal{O}_Q/Q)[T]$, then for each edge $\mathcal{L}$ and for each irreducible factor $\delta_{\mathcal{L},i}(T)$ of $\delta_\mathcal{L}(T)$ there exists exactly one place $P_{\mathcal{L},i}$ of $F'$ lying over $Q$ such that $v_{P_{\mathcal{L},i}}(s) + \rho(\mathcal{L})e(P_{\mathcal{L},i} \mid Q) = 0$ and $\psi_{\mathcal{L},i}(t_Q^{-a}s^b)\in P_{\mathcal{L},i}$, where $a,b\in \mathbb{Z}$, $b>0$, $\mathrm{gcd} \ (a,b) = 1$ and $\rho(\mathcal{L})=-a/b$. Here, with notations as in \cref{thm:newtonarc:splitting}, $\psi_{\mathcal{L},i}(T)\in \mathcal{O}[T]$ is a monic polynomial such that
    \begin{equation*}
        \bar{\psi}_{\mathcal{L},i}(T) = \delta_{\mathcal{L},i}(T) \quad \mbox{and} \quad \mathrm{deg} \ \psi_{\mathcal{L},i}(T) = \mathrm{deg} \ \delta_{\mathcal{L},i}(T).
    \end{equation*}
    Furthermore, the place $P_{\mathcal{L},i}$ satisfies 
    \begin{equation*}
        e(P_{\mathcal{L},i} \mid Q) = b \quad \mbox{and} \quad f(P_{\mathcal{L},i} \mid Q) = \mathrm{deg} \ \delta_{\mathcal{L},i}(T),
    \end{equation*}
    where $e(P_{\mathcal{L},i} \mid Q)$ denotes the ramification index and $f(P_{\mathcal{L},i} \mid Q)$ denotes the relative degree of $P_{\mathcal{L},i}$ over $Q$.
    All places of $F'$ lying over $Q$ are obtained in this way.
\end{corollary}

\subsection{Fibered surfaces}
In this subsection, we collect some definitions on fibered surfaces that provide the essential background for \cref{sec:genus1:fibrations}.

\begin{definition}
    \label[definition]{def:fibration:surfaces}
    Given a smooth, projective, and algebraic surface $X$, a smooth, projective and algebraic curve $B$, a proper morphism with connected general fiber $\pi: X \longrightarrow B$ is called a \emph{fibration}. A morphism $\sigma: B \longrightarrow X$ such that $\pi\circ \sigma= \mathrm{id}_B$, where $\mathrm{id}_B$ is the identity morphism on $B$, is called a \emph{section} of the fibration $\pi$. The curve $B$ is called the \emph{base} of the fibration $\pi$.
\end{definition}

\begin{definition}
Given a fibration $\pi: X \longrightarrow B$, a \emph{multisection for $\pi$} is a geometrically integral curve $C \subseteq X$ such that $\pi|_{C}: C \longrightarrow B$ is finite of degree at least 2 and flat. 
The \emph{degree} of a multisection $C$ is the degree of the map $\pi_{|C}$ above.
\end{definition}

Throughout this subsection, we reserve the letters $X, B, \pi$ for an algebraic surface, an algebraic curve, and a fibration $\pi:X \longrightarrow B$, respectively. 

\begin{example}
Among the simplest class of fibered surfaces are ruled surfaces. These are equipped with a morphism $\pi: X \longrightarrow B$ such that $\pi^{-1}(t)\simeq \mathbb{P}^1$ for all $t\in B$.

If, moreover, the base of the fibration $B\simeq \mathbb{P}^1$, then the surface is rational, i.e., it is birational to $\mathbb{P}^2$. These are called \emph{Hirzebruch surfaces}. There are infinitely many such surfaces, denoted by $\mathbb{F}_n$, where $n$ is the natural number such that there is a unique rational section to the fibration with self-intersection $-n$. 

LRCs on Hirzebruch surfaces were studied in particular in \cite[Section V]{SVAV21}.
\end{example}

One of the most prominent classes of fibered surfaces are \emph{elliptic surfaces}, i.e., surfaces endowed with a fibration of genus one curves that admits a section. Indeed, such a structure allows for a twofold description of the same object, namely via the fibers of the morphism and as an elliptic curve over the function field of the base of the fibration.

\begin{definition}\label[definition]{def: elliptic surface}
A triple $(X,B,\pi)$ is called an \emph{elliptic surface} if the following hold:

\begin{itemize}
\item[a)] $\pi^{-1}(t)$ is a smooth curve of genus 1, for all but finitely many $t \in B$;
\item[b)] there is a section $\sigma: B \longrightarrow X$;
\item[c)] there is at least one singular fiber for $\pi$.
\end{itemize}

The set of sections of $\pi$ is an abelian group which we denote by $\mathrm{MW}(X,\pi)$.
\end{definition}

Note that condition $c)$ implies that ruled surfaces of the type $E \times B$, where $E$ is an elliptic curve, are excluded from our definition of elliptic surface.
If $\pi:X \longrightarrow B$ satisfies condition a), but it does not necessarily admit a section, then it is called a \emph{genus 1 fibration}.

\begin{example}(Elliptic surfaces with a rational base)

Let $(X,\mathbb{P}^1,\pi)$ be an elliptic surface. Then $X$ admits a model with an equation as follows, which is called a \emph{Weierstrass equation}:
\begin{equation}\label{eq:Weierstrass}
y^2+a_1(t,s)xy+a_3(t,s)y=x^3+a_2(t,s)x^2+a_4(t,s)x+ a_6(t,s)
\end{equation}
where $a_i(t,s)$ are homogeneous polynomials of degrees $im$, respectively, and such that at least one $v(a_i(t,s))<i$ at each place of $\mathbb{P}^1$.
\end{example}

If the characteristic of the field of definition is not 2 nor 3 then we can always transform \cref{eq:Weierstrass} in a \emph{short Weierstrass} equation of the form
\begin{equation*}
    y^2=x^3+a(t,s)x+b(t,s),
\end{equation*}
with $\deg a(t,s)= 4m$, $\deg b(t,s)=6m$ and such that $v(a(t,s))<4$ or $v(b(t,s))<6$. 
Given an elliptic surface with an equation in this form, the finitely many singular fibers lie above $[t:s] \in \mathbb{P}^1$ such that $4a(t,s)^3+27b(t,s)^2=0$.

Among surfaces with fibrations, several admit multiple fibrations. As already mentioned in \cref{sec:intro}, this extra structure is also useful in coding theory when constructing LRCs as AG codes from such surfaces. 

\begin{example}
The simplest example of algebraic surface with multiple fibrations is the ruled surface $\mathbb{P}^1\times \mathbb{P}^1$. The two rulings are two evident fibrations on rational curves on the surface.
LRCs with availability on $\mathbb{P}^1$-bundles were constructed in \cite{AAAOAVA24}.
\end{example}

\begin{example}
The only class of surfaces that are not of product type and admit multiple elliptic fibrations are \emph{K3 surfaces}, \cite[Lemma 12.18]{ShiodaSchuett}. For instance, the surface 
\begin{equation}\label{eq: K3 example}
  \E_3: y^2 = x^3 - x^2(t^4+1) + xt^4.
\end{equation}
admits two genus 1 fibrations that are evident from the equation above, namely one over the $t$-line, which we denote by $\Pi_t$, and a second one over the $x$-line, which we denote by $\Pi_x$ (see also \cref{rem:geometric:intuition}, where with a slight abuse we use the same notation for $\Pi_t$ and $\Pi_x$). The former is an elliptic fibration, while the latter is a genus 1 fibration without section and therefore it cannot be put in Weierstrass form over the $x$-line.
\end{example}

\section{Code construction}
\label{sec:code:construction}

In this section, we present a construction of LRCs with availability $2$ from a family of fibered surfaces. We start by introducing the family of surfaces considered, and by discussing how we choose the set of evaluation points for the codes.

Let $r\geq 3$ be an odd integer, let $q$ be a prime power such that $q \equiv 1\pmod{r+1}$, and let $\Fq$ be the finite field with $q$ elements. Let
\begin{equation}
\label{eq:Erdef}
    \mathcal{E}_r: y^2 = x^3 - x^2(t^{r+1}+1) + xt^{r+1}
\end{equation}
and
\begin{equation}
\label{eq:Mrdef}
    \mathcal{M}_r: \begin{cases}
        y^2 = x^3 - x^2(t^{r+1}+1) + xt^{r+1}\\
        y=x^{\frac{r+1}{2}}+1\\
    \end{cases}
\end{equation}
in $\mathbb{A}^2_{(x,y)}\times \mathbb{A}^1_{t}$ over $\F_{q^m}$, where $m\in \mathbb{Z}_{\geq 1}$. In the following discussion, we consider compactifications $\bar{\E}_r$ and $\bar{\M}_r$ of, respectively, $\E_r$ and $\M_r$ in $\P_{[x:y:z]}^2 \times \ \P_{[t:w]}^1$. We denote by $\Fqm(\MM_r)$ the ($\Fqm$-rational) function field of $\MM_r$, which in the affine chart where $z\neq 0$ and $w\neq 0$ can be described as $\Fqm(\MM_r)=\Fqm(t/w,x/z)$.

\begin{remark}
    \label[remark]{rem:functionfield:notation}
    For less cumbersome notations, we henceforth adopt the slightly abusive notation $\Fqm(\MM_r)=\Fqm(t,x)$, in the chart $z\neq 0$ and $w\neq 0$, i.e., we simply write $x$ instead of $x/z$ and $t$ instead of $t/w$.
\end{remark}

\begin{remark}
\label[remark]{rem:characteristic}
    The assumption $r$ odd and $q\equiv 1 \pmod{r+1}$ implies in particular that $\mathrm{char}(\Fqm)\neq 2$. 
\end{remark}

Our goal is to define an evaluation code on $\E_r$ with locality $r$ and availability $\a = 2$, see \cref{def:LRC}. 

Let $\zeta$ be a primitive $r+1$-th root of unity in $\Fqm$. The existence of $\zeta\in \Fqm$ is guaranteed by the assumption $q \equiv 1\pmod{r+1}$, which actually ensures that $\zeta\in \Fq$. To introduce more convenient notations for the discussion below, we start by rewriting \cref{eq:Mrdef} as follows. Let
\begin{equation}
\label{eq:PPt}
    \mathcal{P}_{t}(T):= T^{r+1} + 2T^{\frac{r+1}{2}} - T^3 + T^2(t^{r+1} + 1) - Tt^{r+1} + 1 \in \Fqm(t)[T],
\end{equation}
then \cref{eq:Mrdef} can be rewritten as
\begin{equation}
\label{eq:Mrdef:var}
    \mathcal{M}_r: \begin{cases}
        \mathcal{P}_{t}(x) = 0\\
        y=x^{\frac{r+1}{2}}+1.\\
    \end{cases}
\end{equation}
Note that, for any fixed $\bar{t}\in \Fqm$, the polynomial $\PP_{\bar{t}}(T)$ is an element of $\Fqm[T]$. Moreover, the polynomial $\PP_t(T)$ is irreducible over $\Fqm(t)$. Indeed, let $u:=(-(x^{(r+1)/2} + 1)^2 - x^2 + x^3)/(x(x-1))$. Note first that the function field extension $\Fqm(t,x)/\Fqm(x)$, with defining equation $t^{r+1}=u$, is a Kummer extension of degree $r+1$, which directly follows from the fact that $q \equiv 1\pmod{r+1}$ and that in $\Fqm(x)$ the valuation of $u$ at the place $(x=0)$ is  $v_{(x=0)}(u)=-1$. It is then not difficult to see that 
\begin{equation*}
    \mathrm{deg} \ (t^{r+1})_\infty^{\Fqm(t,x)}= 2(r+1) + r^2 - 1 = (r+1)^2.
\end{equation*}
 Hence, since $[\Fqm(t,x):\Fqm(t^{r+1})]=\mathrm{deg} \ (t^{r+1})_\infty^{\Fqm(t,x)}=(r+1)^2$ and $[\Fqm(t,x):\Fqm(x)]=r+1$, we have that $[\Fqm(x):\Fqm(t^{r+1})]=r+1$, which implies in particular that the polynomial $\PP_t(T)$ defining the extension is irreducible over $\Fqm(t^{r+1})$. With a similar argument, noting that $[\Fqm(t,x):\Fqm(t)]=\mathrm{deg} \ (t)_\infty^{\Fqm(t,x)}=r+1$, it also follows that $\PP_t(T)$ is irreducible over $\Fqm(t)$. See \cref{fig:subfields} for a diagram summarising this discussion.

\begin{figure}[ht]
\adjustbox{center}{
\begin{tikzcd}
                                             & {\mathbb{F}_{q^m}(t,x)} \arrow[rdd, "r+1", no head] \arrow[ldd, "r+1"', no head] &                                             \\
                                             &                                                                              &                                             \\
\mathbb{F}_{q^m}(t) \arrow[rdd, "r+1"', no head] &                                                                              & \mathbb{F}_{q^m}(x) \arrow[ldd, "r+1", no head] \\
                                             &                                                                              &                                             \\
                                             & \mathbb{F}_{q^m}(t^{r+1})                                                        &                                            
\end{tikzcd}
}
\caption{Some subfields of $\Fqm(t,x)$. \label{fig:subfields}}
\end{figure}
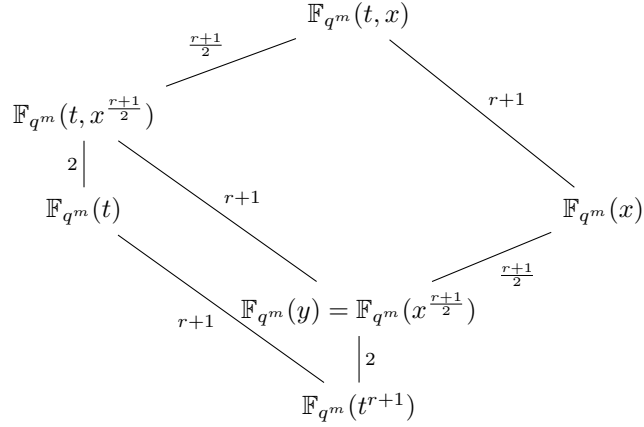

With the notations just introduced above, let now $\bar{\E}_r$ and $\bar{\M}_r$ denote, respectively, the compactification of $\E_r$ and $\M_r$ in $\P_{[x:y:z]}^2 \times \ \P_{[t:w]}^1$.

For a point $P:=[\bar{x}:\bar{x}^{\frac{r+1}{2}}+1:1;\bar{t}:1]\in \MM_r(\Fqm)$ we define two associated subsets of $\MM_r(\Fqm)$
\begin{align}
\begin{split}
\label{eq:recsubsets:def}
    H_P&:=\{[\bar{x}:\bar{x}^{\frac{r+1}{2}}+1:1;\zeta^i\bar{t}:1]\in \MM_r(\Fqm)\mid \ i=1,\ldots,r\}, \\ V_P&:=\{[x:x^{\frac{r+1}{2}}+1:1;\bar{t}:1]\in \MM_r(\Fqm) \mid \ \PP_{\bar{t}}(x)=0 \ \mbox{and} \ x\neq \bar{x}\}.
\end{split}
\end{align}

\begin{definition}
    \label[definition]{defn:nicepoint}
    A point $P:=[\bar{x}:\bar{x}^{\frac{r+1}{2}}+1:1;\bar{t}:1]\in \MM_r(\Fqm)$ is said to be \emph{nice} if its associated subsets $H_P$ and $V_P$ defined in \cref{eq:recsubsets:def} satisfy
    \begin{equation*}
    |H_P|=|V_P|=r.
\end{equation*}
\end{definition}
As evaluation points of the code, we wish to select a subset of affine $\Fqm$-rational points of $\MM_r$ that are \emph{nice}.

Observe that, by Kummer theory, from \cref{eq:Mrdef} it follows immediately that if $\bar{x}\neq 0,1$ and $(\bar{x}^{(r+1)/2} + 1)^2 + \bar{x}^2 - \bar{x}^3 \neq 0$, then $|H_P|=r$. On the other hand, it is less straightforward to understand when $|V_P|=r$. Indeed, this is equivalent to studying the problem of whether, given $\bar{t}\in \Fqm$, the polynomial 
\begin{equation*}
    \mathcal{P}_{\bar{t}}(T):= T^{r+1} + 2T^{\frac{r+1}{2}} - T^3 + T^2(\bar{t}^{r+1} + 1) - T\bar{t}^{r+1} + 1 \in \Fqm[T]
\end{equation*}
splits into $r+1$ distinct linear factors over $\Fqm$. The following lemma gives a sufficient condition on $m$ to guarantee the existence of at least one $\bar{t}\in \Fqm\setminus\{0\}$ such that $\mathcal{P}_{\bar{t}}(T)$ splits completely over $\Fqm$. Note that we disregard $\bar{t} = 0$ because it would correspond to points for which $|H_P| < r$. 

\begin{lemma}
    \label[lemma]{lemma:existence:of:evalpoints}
    Let $q$, $r$, $\M_r$ and $\PP_t(T)$ as defined above. For fixed $q$ and $r$, let $\mathfrak{M}\in \mathbb{Z}_{\geq 1}$ be such that the following inequalities are satisfied: 
    \begin{align}
    \label{eq:m:condition}
    \begin{cases}
        q^\mathfrak{M} \geq 2g(E)\left(q^{\mathfrak{M}/2} + 1\right),\\
        \frac{q^\mathfrak{M}}{\mathfrak{M}(r+1)!} - \frac{2}{\mathfrak{M}(r+1)!} \cdot \left(((r+1)! + g(E))\cdot q^{\mathfrak{M}/2} + (r+1)! \cdot q^{\mathfrak{M}/4} + g(E) + (r+1)!\right) + 1 \geq 3,
    \end{cases}
    \end{align}
    where $E$ is the Galois closure of $\F_{q^{\mathfrak{M}}}(x)/\F_{q^{\mathfrak{M}}}(t^{r+1})$ and $g(E)$ is the genus of $E$. Then, 
    there exists at least one place $\mathfrak{P}$ of $\F_{q^{\mathfrak{M}}}(t)$, different from the zero and the pole of $t$, that splits completely in the extension $\F_{q^{\mathfrak{M}}}(t,x)/\F_{q^{\mathfrak{M}}}(t)$. Moreover, let $\mathfrak{n} := \mathrm{deg} \ \mathfrak{P} \geq 1$ and let $m := \mathfrak{M} \cdot \mathfrak{n}$. Then there exists
    $\bar{t}\in \Fqm \setminus\{0\}$ such that $\mathcal{P}_{\bar{t}}(T)$ splits into $r+1$ distinct factors over $\Fqm$.
\end{lemma}
\begin{proof}
     Let $(t^{r+1}=a)$, $a\in \F_{q^\mathfrak{M}} \cup \{\infty\}$, be any rational place of $\F_{q^{\mathfrak{M}}}(t^{r+1})$, and consider the function field extension $\F_{q^{\mathfrak{M}}}(x)/\F_{q^{\mathfrak{M}}}(t^{r+1})$, where $\PP_t(x)=0$ is the defining equation. 

    We start by proving the first part of the lemma. To this aim, we show that, if \cref{eq:m:condition} is satisfied, then there exists a place $(t^{r+1}=a)$, with $a\in \F_{q^\mathfrak{M}} \setminus \{0\}$, which splits completely in $\F_{q^{\mathfrak{M}}}(x)/\F_{q^{\mathfrak{M}}}(t^{r+1})$. Indeed, by \cite[Proposition 3.9.6]{Sti} (see \cref{fig:subfields}), this implies that every extension $\mathfrak{P} \mid (t^{r+1}=a)$ in $\F_{q^{\mathfrak{M}}}(t)$ splits completely in $\F_{q^{\mathfrak{M}}}(t,x)/\F_{q^{\mathfrak{M}}}(t)$.
     
    Let $E$ be the Galois closure of $\F_{q^{\mathfrak{M}}}(x)/\F_{q^{\mathfrak{M}}}(t^{r+1})$. By the Chebotarev Density Theorem (see for instance \cite[Theorem 9.13B]{Rosen}), for each $\mathfrak{M}\in \mathbb{Z}_{\geq 1}$ the cardinality of the set of rational places of $\F_{q^{\mathfrak{M}}}(t^{r+1})$ that split completely in $E$ is $q^\mathfrak{M}/(\mathfrak{M}\cdot |G|) + O(q^{\mathfrak{M}/2}/\mathfrak{M})$, where $G$ is the Galois group of $E/\F_{q^{\mathfrak{M}}}(t^{r+1})$. We consider the explicit estimate of the error term $O(q^{\mathfrak{M}/2}/\mathfrak{M})$ in \cite[Proposition 7.4.8]{FJ05}, from which the following lower bound for the number $N$ of rational places that split completely in $E/\F_{q^{\mathfrak{M}}}(t^{r+1})$ holds:
        \begin{equation}
        \label{eq:N}
            N \geq \frac{q^\mathfrak{M}}{\mathfrak{M}|G|} - \frac{2}{\mathfrak{M}|G|} \cdot \left((|G| + g(E))\cdot q^{\mathfrak{M}/2} + |G| \cdot q^{\mathfrak{M}/4} + g(E) + |G|\right) + 1.
        \end{equation}

    Combining the fact that $|G|\leq (r+1)!$ with the first inequality in \cref{eq:m:condition}, we have that
    \begin{align*}
        \frac{q^\mathfrak{M}}{\mathfrak{M}|G|} - \frac{2}{\mathfrak{M}|G|} \cdot \left((|G| + g(E))\cdot q^{\mathfrak{M}/2} + |G| \cdot q^{\mathfrak{M}/4} + g(E) + |G|\right) \\ \geq \frac{q^\mathfrak{M}}{\mathfrak{M}(r+1)!} - \frac{2}{\mathfrak{M}(r+1)!} \cdot \left(((r+1)! + g(E))\cdot q^{\mathfrak{M}/2} + (r+1)! \cdot q^{\mathfrak{M}/4} + g(E) + (r+1)!\right).
    \end{align*}
    Hence, from the second inequality in \cref{eq:m:condition} and from \cref{eq:N} we obtain
    \begin{equation*}
        N \geq \frac{q^\mathfrak{M}}{\mathfrak{M}|G|} - \frac{2}{\mathfrak{M}|G|} \cdot \left((|G| + g(E))\cdot q^{\mathfrak{M}/2} + |G| \cdot q^{\mathfrak{M}/4} + g(E) + |G|\right) + 1 \geq 3,
    \end{equation*}
    which guarantees that there is at least one rational place $\mathfrak{p}:=(t^{r+1}=a)$ of $\F_{q^{\mathfrak{M}}}(t^{r+1})$ different from the zero and the pole of $t$ (i.e., $a\in \F_{q^\mathfrak{M}} \setminus \{0\}$) which is completely splitting in $E/\F_{q^{\mathfrak{M}}}(t^{r+1})$. By the observations above, this proves the first part of the statement, since any place of $\F_{q^{\mathfrak{M}}}(t^{r+1})$ which splits completely in $E/\F_{q^{\mathfrak{M}}}(t^{r+1})$ must also split completely in $\F_{q^{\mathfrak{M}}}(x)/\F_{q^{\mathfrak{M}}}(t^{r+1})$.

    Let $\mathfrak{P}$ be an extension of $\mathfrak{p}$ in $\F_{q^{\mathfrak{M}}}(t)/\F_{q^{\mathfrak{M}}}(t^{r+1})$. By what we have just shown and by the discussion above, $\mathfrak{P}$ splits completely in $\F_{q^{\mathfrak{M}}}(t,x)/\F_{q^{\mathfrak{M}}}(t)$.
    To prove the second part of the lemma, it is now sufficient to note the following. If $\mathfrak{n} := \mathrm{deg} \ \mathfrak{P}$ and $m := \mathfrak{M} \cdot \mathfrak{n}$, then $\mathfrak{P}$ splits completely in the degree $\mathfrak{n}$ constant field extension $\Fqm\F_{q^{\mathfrak{M}}}(t)/\F_{q^{\mathfrak{M}}}(t) = \Fqm(t)/\F_{q^{\mathfrak{M}}}(t)$ (see \cite[Lemma 5.1.9]{Sti}). This hence implies that there is a rational place of $\Fqm(t)$, namely one of the $\mathfrak{n}$ extensions $\mathfrak{Q}_1, \ldots, \mathfrak{Q}_{\mathfrak{n}}$ of $\mathfrak{P}$ in $\Fqm(t)/\F_{q^{\mathfrak{M}}}(t)$, that splits completely in $\Fqm(t,x)/\Fqm(t)$ (in fact, all the $\mathfrak{n}$ extensions of $\mathfrak{P}$ split completely in $\Fqm(t,x)/\Fqm(t)$). To conclude, it suffices to observe that this is precisely equivalent to the existence of $\bar{t}\in \Fqm \setminus\{0\}$ such that $\mathcal{P}_{\bar{t}}(T)$ splits into $r+1$ distinct factors over $\Fqm$, since every rational place of $\Fqm(t)$ different from the zero and the pole of $t$ is of the form $(t = \bar{t})$, for some $\bar{t}\in \Fqm\setminus\{0\}$.
\end{proof}

\begin{definition}
    \label[definition]{defn:niceelement}
    An element $\bar{t}\in \Fqm$ is said to be \emph{nice} if $\bar{t}\neq 0$ and the polynomial $\PP_{\bar{t}}(T)$ (see \cref{eq:PPt}) splits completely into distinct factors over $\Fqm$.
\end{definition}

\begin{remark}
\label[remark]{rem:m:upperbound}
    \cref{lemma:existence:of:evalpoints} gives a sufficient condition to show that it is always possible, for any fixed $q$ and $r$, to choose a field of definition $\Fqm$ of $\E_r$ and $\M_r$ such that there exists at least one nice $\bar{t}\in \Fqm \setminus\{0\}$. However, from the lemma it is not a priori clear how large the smallest $m\in \mathbb{Z}_{\geq 1}$ satisfying \cref{eq:m:condition} is, since the inequalities in \cref{eq:m:condition} require to estimate the genus $g(E)$, which is, in general, not a straightforward task. Nevertheless, nice elements exist also in fields of much smaller cardinality than those satisfying \cref{eq:m:condition}, and they can be found effectively by direct computation, for instance with the aid of a computer algebra system. In fact, in all the examples in \cref{table:codes}, the nice elements were directly computed in this way.
\end{remark}

\emph{Nice} elements of $\Fqm$ (see \cref{defn:niceelement}) and \emph{nice} points in $\MM_r(\Fqm)$ (see \cref{defn:nicepoint}) are closely related. 

Indeed, note first that the existence of a nice $\bar{t}$ directly implies the existence of $r+1$ \emph{nice} elements of $\Fqm$, namely $\bar{t}, \zeta \bar{t},\ldots,\zeta^r\bar{t}$, since $\PP_{\bar{t}}(T) = \PP_{\zeta^i\bar{t}}(T)$ for all $i=1,\ldots,r$.
This means that the existence of a nice $\bar{t}\in \Fqm$ ensures the existence of $(r+1)^2$ nice points $P$ of $\MM_r$. 
Indeed, any nice $\bar{t}\in \Fqm$ directly yields $r+1$ nice points $P_i:=[\bar{x}_i:\bar{x}_i^{\frac{r+1}{2}}+1:1;\bar{t}:1]\in \MM_r(\Fqm)$, $i=1,\ldots,r+1$, where $\PP_{\bar{t}}(\bar{x}_i)=0$. Since, for any $i$, $V_{P_i}=\bigcup_{j\neq i}\{P_j\}$, it is immediate to see that $|V_{P_i}|=r$ for all $i$. On the other hand, for any $i=1,\ldots, r+1$, each point in the set $H_{P_i}$, is also a nice point since $\PP_{\bar{t}}(T) = \PP_{\zeta^i\bar{t}}(T)$, for all $i=1,\ldots,r$. This means that, for any nice $\bar{t}\in \Fqm$, we have the following set of $(r+1)^2$ nice points of $\MM_r(\Fqm)$:
\begin{align}
\label{eq:Sbar:def}
    S_{\bar{t}}:=\{[x: x^{\frac{r+1}{2}} + 1 : 1 ; \zeta^j\bar{t} : 1 ] \mid j=0,\ldots, r, \ \PP_{\bar{t}}(x)=0\}.  
\end{align}
Henceforth, we denote by 
\begin{align}
\label{eq:Gmdef}
    \mathcal{G}_m := \left\{\bar{t}\in \Fqm\setminus\{0\} \mid \  \PP_{\bar{t}}(T)=\prod_{\substack{i=1\\ \alpha_i\in \Fqm \ \forall i, \\ \forall i \neq j, \ \alpha_i \neq \alpha_j,}}^{r+1}(T-\alpha_i) \right\} \subseteq \Fqm,
\end{align}
i.e., the subset of elements of $\Fqm$ which are \emph{nice}. Note that $|\mathcal{G}_m|=(r+1)M$, for some $M\in \mathbb{Z}_{\geq 1}$, since if $\bar{t}\in \Fqm$ is nice then also $\zeta^j\bar{t}\in \Fqm$ is nice, for all $j=1,\ldots,r$. This means that $\mathcal{G}_m$ can be partitioned into $M$ many subsets
\begin{equation}
\label{eq:Gmdef:2}
    \begin{aligned}
    \mathcal{G}_m &= \bigcup_{\ell=1}^M \mathcal{G}_m^{(\ell)}\\
    &= \bigcup_{\ell=1}^M\{\zeta^j\bar{t}^{(\ell)} \mid j=1,\ldots, r+1\},
    \end{aligned}
\end{equation}
such that for any $\ell, k \in \{1,\ldots, M\}$ there exist $j_1, j_2\in \{1,\ldots,r+1\}$ with $\zeta^{j_1}\bar{t}^{(\ell)} = \zeta^{j_2}\bar{t}^{(k)}$ if and only if $\ell=k$.

\begin{definition}
\label[definition]{def:evalpoints}
    Let $q$ and $r$ be as above, and let $m\in \mathbb{Z}_{\geq 1}$ be a positive integer such that $\mathcal{G}_m \neq \emptyset$.  Let $\E_r$ and $\M_r$ be defined over $\Fqm$, and choose subsets $\mathcal{G}_m^{(1)},\ldots,  \mathcal{G}_m^{(b)}$ of $\mathcal{G}_m$, where $b\in \mathbb{Z}$, $1\leq b \leq M$. Let $B:=\cup_{\ell=1}^b \mathcal{G}_m^{(\ell)}$.
    
    We define the set
    \begin{align*}
    S_B&:=\bigcup_{\bar{t}\in B} S_{\bar{t}}\\
    &=\{  P_{i,j}^{(\ell)}:= [\bar{x}_i^{(\ell)} : \left(\bar{x}_i^{(\ell)}\right)^{\frac{r+1}{2}} + 1 : 1 ; \bar{t}_j^{(\ell)} : 1 ]\in \MM_r(\Fqm) \mid \ell=1,\ldots, b, \ i=1,\ldots, r+1, \ j=1,\ldots, r+1 \}, 
    \end{align*}
    where, for $\ell$ fixed, $\bar{t}_j^{(\ell)}:=\zeta^j\bar{t}^{(\ell)}$.
    Then we have that $|S_B|=b(r+1)^2$.

    With these notations in place, we can conveniently redefine for any point $P_{\bar{i},\bar{j}}^{(\bar{\ell})}\in S_B$ the subsets introduced in \cref{eq:recsubsets:def} as
    \begin{equation}
\label{eq:horizontalset}
    H_{\bar{i}}^{(\bar{\ell})}:=\left\{P_{\bar{i},j}^{(\bar{\ell})} \mid j=1,\ldots,r+1 \ \mbox{and} \ j\neq \bar{j}\right\}
\end{equation}
and
\begin{equation}
\label{eq:verticalset}
    V_{\bar{j}}^{(\bar{\ell})}:=\left\{P_{i,\bar{j}}^{(\bar{\ell})} \mid i=1,\ldots,r+1 \ \mbox{and} \  i\neq \bar{i}\right\}.
\end{equation}
\end{definition}

\begin{remark}
    \label[remark]{rem:geometric:intuition}
    There are multiple ways to interpret the \emph{nice} points in $\MM_r(\Fqm)$. Indeed, on the one hand, nice points correspond exactly to rational places of $\Fqm(t,x)$ that lie over a completely splitting place both in $\Fqm(t,x)/\Fqm(t)$ and $\Fqm(t,x)/\Fqm(x)$. The function field extensions $\Fqm(t,x)/\Fqm(t)$ and $\Fqm(t,x)/\Fqm(x)$ can be seen as obtained from the pullback of rational maps $\pi_t: \MM_r \longrightarrow \P^1_{[t:w]}$ and $\pi_x: \MM_r \longrightarrow \P^1_{[x:z]}$, respectively, where $\pi_t : [\bar{x}:\bar{x}^{\frac{r+1}{2}}+1:1;\bar{t}:1] \longmapsto \t$  is the projection onto the $t$-coordinate and $\pi_x: [\bar{x}:\bar{x}^{\frac{r+1}{2}}+1:1;\bar{t}:1] \longmapsto \bar{x}$ it the projection onto the $x$-coordinate. 
    In this setting, given a \emph{nice} point $P:=[\bar{x}:\bar{x}^{\frac{r+1}{2}}+1:1;\bar{t}:1]\in \MM_r(\Fqm)$, we hence have that $H_P=\pi_x^{-1}([\bar{x}:1])\setminus \{P\}$ and $V_P=\pi_t^{-1}([\t:1])\setminus \{P\}$. Moreover, the set $\mathcal{G}_m$ corresponds to the set of affine points of $\P^1_{[t:w]}$ whose preimage under $\pi_t$ consists of exactly $r+1$ $\Fqm$-rational points of $\MM_r$.  

    On the other hand, one can consider the two fibrations
    \begin{align*}
    \begin{split}
    \Pi_t: \EE_r \longrightarrow \P^1_{[t:w]},\\
    \Pi_x: \EE_r \longrightarrow \P^1_{[x:z]},
    \end{split}
    \end{align*}
    where, similarly as above, $\Pi_t,\Pi_x$ are the projection onto the $t$-coordinate, resp. the $x$-coordinate. Observe that such maps are in fact proper morphisms with connected general fiber. 
    Given a \emph{nice} point $P\in \MM_r(\Fqm)$, note that it can be also regarded as a point of $\EE_r(\Fqm)$. This inspires the notation for the sets $H_P$ and $V_P$ introduced in \cref{eq:recsubsets:def} (see also \cref{eq:horizontalset} and \cref{eq:verticalset}). We call the fibration $\Pi_t$ \emph{vertical} and the fibration $\Pi_x$ \emph{horizontal}, hence since $H_P=\Pi_x^{-1}([\bar{x}:1])\setminus \{P\}$ and $V_P=\Pi_t^{-1}([\t:1])\setminus \{P\}$, in what follows we refer to $H_P$ as the \emph{horizontal set} associated to $P$ and as $V_P$ as the \emph{vertical set} associated to $P$.

    In our code construction, this terminology comes in handy to convey the intuition behind the recovery procedure outlined in \cref{lemma:genconstruction:recovery}. See \cref{fig:visualintuition} for a visual intuition of a set $S_{\bar{t}}$ of points as in \cref{eq:Sbar:def}.

    Note that, with slight abuse of notation, in the following sections we simply write $\Pi_t^{-1}(\bar{t})$ (resp. $\Pi_x^{-1}(\bar{x})$) for the fiber over the point $[\bar{t}:1]$ (resp. $[\bar{x}:1]$), and we refer to $\Pi_t^{-1}(\bar{t})$ as a \emph{vertical} (resp. $\Pi_x^{-1}(\bar{x})$ as a \emph{horizontal}) fiber. 
\end{remark}

\begin{figure}[ht]
    \centering
\adjustbox{scale=0.8}{%

\tikzset{every picture/.style={line width=0.75pt}} 

\begin{tikzpicture}[x=0.75pt,y=0.75pt,yscale=-1,xscale=1]

\draw  [color={rgb, 255:red, 208; green, 2; blue, 27 }  ,draw opacity=1 ][line width=1.5] [line join = round][line cap = round] (79.5,118.5) .. controls (74.45,127.33) and (71.16,138.16) .. (79.5,146.5) .. controls (82.52,149.52) and (87.3,150.84) .. (89.5,154.5) .. controls (96.49,166.14) and (99.95,174.76) .. (102.5,187.5) .. controls (105.43,202.14) and (82.69,223.47) .. (85.5,237.5) .. controls (87.21,246.07) and (97.31,251.58) .. (102.5,258.5) .. controls (106.23,263.47) and (105.77,278.1) .. (105.5,282.5) .. controls (104.54,297.92) and (82.33,312) .. (84.5,331.5) .. controls (85.82,343.37) and (109.29,343.53) .. (106.5,357.5) .. controls (105.46,362.69) and (98.4,362.6) .. (95.5,365.5) ;
\draw  [color={rgb, 255:red, 208; green, 2; blue, 27 }  ,draw opacity=1 ][line width=1.5] [line join = round][line cap = round] (198.5,116.5) .. controls (193.45,125.33) and (190.16,136.16) .. (198.5,144.5) .. controls (201.52,147.52) and (206.3,148.84) .. (208.5,152.5) .. controls (215.49,164.14) and (218.95,172.76) .. (221.5,185.5) .. controls (224.43,200.14) and (201.69,221.47) .. (204.5,235.5) .. controls (206.21,244.07) and (216.31,249.58) .. (221.5,256.5) .. controls (225.23,261.47) and (224.77,276.1) .. (224.5,280.5) .. controls (223.54,295.92) and (201.33,310) .. (203.5,329.5) .. controls (204.82,341.37) and (228.29,341.53) .. (225.5,355.5) .. controls (224.46,360.69) and (217.4,360.6) .. (214.5,363.5) ;
\draw  [color={rgb, 255:red, 208; green, 2; blue, 27 }  ,draw opacity=1 ][line width=1.5] [line join = round][line cap = round] (317.5,115.5) .. controls (312.45,124.33) and (309.16,135.16) .. (317.5,143.5) .. controls (320.52,146.52) and (325.3,147.84) .. (327.5,151.5) .. controls (334.49,163.14) and (337.95,171.76) .. (340.5,184.5) .. controls (343.43,199.14) and (320.69,220.47) .. (323.5,234.5) .. controls (325.21,243.07) and (335.31,248.58) .. (340.5,255.5) .. controls (344.23,260.47) and (343.77,275.1) .. (343.5,279.5) .. controls (342.54,294.92) and (320.33,309) .. (322.5,328.5) .. controls (323.82,340.37) and (347.29,340.53) .. (344.5,354.5) .. controls (343.46,359.69) and (336.4,359.6) .. (333.5,362.5) ;
\draw  [color={rgb, 255:red, 208; green, 2; blue, 27 }  ,draw opacity=1 ][line width=1.5] [line join = round][line cap = round] (439.5,114.5) .. controls (434.45,123.33) and (431.16,134.16) .. (439.5,142.5) .. controls (442.52,145.52) and (447.3,146.84) .. (449.5,150.5) .. controls (456.49,162.14) and (459.95,170.76) .. (462.5,183.5) .. controls (465.43,198.14) and (442.69,219.47) .. (445.5,233.5) .. controls (447.21,242.07) and (457.31,247.58) .. (462.5,254.5) .. controls (466.23,259.47) and (465.77,274.1) .. (465.5,278.5) .. controls (464.54,293.92) and (442.33,308) .. (444.5,327.5) .. controls (445.82,339.37) and (469.29,339.53) .. (466.5,353.5) .. controls (465.46,358.69) and (458.4,358.6) .. (455.5,361.5) ;
\draw  [color={rgb, 255:red, 74; green, 144; blue, 226 }  ,draw opacity=1 ][line width=1.5] [line join = round][line cap = round] (57.5,129.5) .. controls (63.95,129.5) and (89.32,125.44) .. (90.5,132.5) .. controls (92.7,145.72) and (77.18,140.71) .. (71.5,144.5) .. controls (67.69,147.04) and (64.9,152.72) .. (71.5,153.5) .. controls (85.44,155.14) and (93.89,150.3) .. (106.5,148.5) .. controls (123.5,146.07) and (139.69,143.08) .. (156.5,138.5) .. controls (172.1,134.24) and (186.87,128.22) .. (203.5,129.5) .. controls (207.1,129.78) and (206.8,134.9) .. (205.5,137.5) .. controls (205.35,137.8) and (204.74,137.26) .. (204.5,137.5) .. controls (201.74,140.26) and (178.27,148.39) .. (184.5,151.5) .. controls (192.53,155.52) and (218.75,147.59) .. (227.5,146.5) .. controls (244.44,144.38) and (261.59,143.88) .. (278.5,140.5) .. controls (294.18,137.36) and (310.91,131.22) .. (327.5,132.5) .. controls (329.55,132.66) and (330.38,138.74) .. (329.5,140.5) .. controls (327.07,145.36) and (302.27,142.8) .. (305.5,152.5) .. controls (306.61,155.83) and (330.26,150.88) .. (333.5,150.5) .. controls (355.3,147.93) and (376.89,147.9) .. (398.5,145.5) .. controls (409.84,144.24) and (421.16,139.07) .. (432.5,138.5) .. controls (438.49,138.2) and (444.51,138.18) .. (450.5,138.5) .. controls (453.96,138.68) and (459.71,143.89) .. (454.5,146.5) .. controls (453.9,146.8) and (453.1,146.2) .. (452.5,146.5) .. controls (449.56,147.97) and (428.57,149.63) .. (431.5,155.5) .. controls (433.33,159.16) and (442.04,158.77) .. (444.5,158.5) .. controls (461.83,156.57) and (478.41,151.5) .. (495.5,151.5) ;
\draw  [color={rgb, 255:red, 74; green, 144; blue, 226 }  ,draw opacity=1 ][line width=1.5] [line join = round][line cap = round] (66.5,206.5) .. controls (72.95,206.5) and (98.32,202.44) .. (99.5,209.5) .. controls (101.7,222.72) and (86.18,217.71) .. (80.5,221.5) .. controls (76.69,224.04) and (73.9,229.72) .. (80.5,230.5) .. controls (94.44,232.14) and (102.89,227.3) .. (115.5,225.5) .. controls (132.5,223.07) and (148.69,220.08) .. (165.5,215.5) .. controls (181.1,211.24) and (195.87,205.22) .. (212.5,206.5) .. controls (216.1,206.78) and (215.8,211.9) .. (214.5,214.5) .. controls (214.35,214.8) and (213.74,214.26) .. (213.5,214.5) .. controls (210.74,217.26) and (187.27,225.39) .. (193.5,228.5) .. controls (201.53,232.52) and (227.75,224.59) .. (236.5,223.5) .. controls (253.44,221.38) and (270.59,220.88) .. (287.5,217.5) .. controls (303.18,214.36) and (319.91,208.22) .. (336.5,209.5) .. controls (338.55,209.66) and (339.38,215.74) .. (338.5,217.5) .. controls (336.07,222.36) and (311.27,219.8) .. (314.5,229.5) .. controls (315.61,232.83) and (339.26,227.88) .. (342.5,227.5) .. controls (364.3,224.93) and (385.89,224.9) .. (407.5,222.5) .. controls (418.84,221.24) and (430.16,216.07) .. (441.5,215.5) .. controls (447.49,215.2) and (453.51,215.18) .. (459.5,215.5) .. controls (462.96,215.68) and (468.71,220.89) .. (463.5,223.5) .. controls (462.9,223.8) and (462.1,223.2) .. (461.5,223.5) .. controls (458.56,224.97) and (437.57,226.63) .. (440.5,232.5) .. controls (442.33,236.16) and (451.04,235.77) .. (453.5,235.5) .. controls (470.83,233.57) and (487.41,228.5) .. (504.5,228.5) ;
\draw  [color={rgb, 255:red, 74; green, 144; blue, 226 }  ,draw opacity=1 ][line width=1.5] [line join = round][line cap = round] (78.5,270.5) .. controls (84.95,270.5) and (110.32,266.44) .. (111.5,273.5) .. controls (113.7,286.72) and (98.18,281.71) .. (92.5,285.5) .. controls (88.69,288.04) and (85.9,293.72) .. (92.5,294.5) .. controls (106.44,296.14) and (114.89,291.3) .. (127.5,289.5) .. controls (144.5,287.07) and (160.69,284.08) .. (177.5,279.5) .. controls (193.1,275.24) and (207.87,269.22) .. (224.5,270.5) .. controls (228.1,270.78) and (227.8,275.9) .. (226.5,278.5) .. controls (226.35,278.8) and (225.74,278.26) .. (225.5,278.5) .. controls (222.74,281.26) and (199.27,289.39) .. (205.5,292.5) .. controls (213.53,296.52) and (239.75,288.59) .. (248.5,287.5) .. controls (265.44,285.38) and (282.59,284.88) .. (299.5,281.5) .. controls (315.18,278.36) and (331.91,272.22) .. (348.5,273.5) .. controls (350.55,273.66) and (351.38,279.74) .. (350.5,281.5) .. controls (348.07,286.36) and (323.27,283.8) .. (326.5,293.5) .. controls (327.61,296.83) and (351.26,291.88) .. (354.5,291.5) .. controls (376.3,288.93) and (397.89,288.9) .. (419.5,286.5) .. controls (430.84,285.24) and (442.16,280.07) .. (453.5,279.5) .. controls (459.49,279.2) and (465.51,279.18) .. (471.5,279.5) .. controls (474.96,279.68) and (480.71,284.89) .. (475.5,287.5) .. controls (474.9,287.8) and (474.1,287.2) .. (473.5,287.5) .. controls (470.56,288.97) and (449.57,290.63) .. (452.5,296.5) .. controls (454.33,300.16) and (463.04,299.77) .. (465.5,299.5) .. controls (482.83,297.57) and (499.41,292.5) .. (516.5,292.5) ;
\draw  [color={rgb, 255:red, 74; green, 144; blue, 226 }  ,draw opacity=1 ][line width=1.5] [line join = round][line cap = round] (68.5,318.5) .. controls (74.95,318.5) and (100.32,314.44) .. (101.5,321.5) .. controls (103.7,334.72) and (88.18,329.71) .. (82.5,333.5) .. controls (78.69,336.04) and (75.9,341.72) .. (82.5,342.5) .. controls (96.44,344.14) and (104.89,339.3) .. (117.5,337.5) .. controls (134.5,335.07) and (150.69,332.08) .. (167.5,327.5) .. controls (183.1,323.24) and (197.87,317.22) .. (214.5,318.5) .. controls (218.1,318.78) and (217.8,323.9) .. (216.5,326.5) .. controls (216.35,326.8) and (215.74,326.26) .. (215.5,326.5) .. controls (212.74,329.26) and (189.27,337.39) .. (195.5,340.5) .. controls (203.53,344.52) and (229.75,336.59) .. (238.5,335.5) .. controls (255.44,333.38) and (272.59,332.88) .. (289.5,329.5) .. controls (305.18,326.36) and (321.91,320.22) .. (338.5,321.5) .. controls (340.55,321.66) and (341.38,327.74) .. (340.5,329.5) .. controls (338.07,334.36) and (313.27,331.8) .. (316.5,341.5) .. controls (317.61,344.83) and (341.26,339.88) .. (344.5,339.5) .. controls (366.3,336.93) and (387.89,336.9) .. (409.5,334.5) .. controls (420.84,333.24) and (432.16,328.07) .. (443.5,327.5) .. controls (449.49,327.2) and (455.51,327.18) .. (461.5,327.5) .. controls (464.96,327.68) and (470.71,332.89) .. (465.5,335.5) .. controls (464.9,335.8) and (464.1,335.2) .. (463.5,335.5) .. controls (460.56,336.97) and (439.57,338.63) .. (442.5,344.5) .. controls (444.33,348.16) and (453.04,347.77) .. (455.5,347.5) .. controls (472.83,345.57) and (489.41,340.5) .. (506.5,340.5) ;
\draw  [color={rgb, 255:red, 74; green, 74; blue, 74 }  ,draw opacity=1 ][dash pattern={on 0.84pt off 2.51pt}][line width=0.75] [line join = round][line cap = round] (66.45,139.6) .. controls (76.05,139.6) and (85.72,143.87) .. (91.45,149.6) .. controls (93.5,151.65) and (101.67,152.67) .. (104.45,153.6) .. controls (104.88,153.74) and (107.13,156.72) .. (108.45,157.6) .. controls (117.88,163.89) and (128.08,174.34) .. (121.45,187.6) .. controls (119.84,190.81) and (115.94,194.36) .. (113.45,195.6) .. controls (109.81,197.42) and (108.08,203.28) .. (105.45,204.6) .. controls (97.78,208.43) and (88.95,209.2) .. (80.45,212.6) .. controls (73.12,215.53) and (60.03,232.53) .. (62.45,240.6) .. controls (65.2,249.78) and (75.51,250.44) .. (81.45,254.6) .. controls (90.64,261.04) and (95.78,267.39) .. (105.45,268.6) .. controls (124.83,271.02) and (128.3,264.36) .. (130.45,281.6) .. controls (130.74,283.94) and (132.07,286.33) .. (131.45,288.6) .. controls (131.33,289.04) and (127.59,294.33) .. (127.45,294.6) .. controls (118.24,313.02) and (116.1,312.87) .. (88.45,314.6) .. controls (86.51,314.72) and (79.96,323.47) .. (78.45,324.6) .. controls (71.71,329.65) and (68.57,341.87) .. (67.45,348.6) .. controls (67.22,349.97) and (68.06,356.52) .. (69.45,356.6) .. controls (85.65,357.55) and (88.1,357.56) .. (104.45,356.6) .. controls (107.78,356.4) and (112,353.7) .. (115.45,353.6) .. controls (129.69,353.18) and (143.34,355.64) .. (155.45,351.6) .. controls (160.39,349.95) and (163.71,345.44) .. (168.45,342.6) .. controls (171.32,340.88) and (176.45,340.1) .. (179.45,338.6) .. controls (183.58,336.53) and (184.15,331.9) .. (186.45,329.6) .. controls (190.8,325.25) and (197.34,326.71) .. (201.45,322.6) .. controls (203.49,320.56) and (206.08,314.79) .. (208.45,313.6) .. controls (211.71,311.97) and (223.76,306.84) .. (225.45,306.6) .. controls (228.56,306.16) and (238.88,305.57) .. (239.45,301.6) .. controls (241.08,290.17) and (231.95,288.1) .. (226.45,282.6) .. controls (224.36,280.51) and (223.51,274.62) .. (220.45,273.6) .. controls (205.58,268.64) and (197.08,258.09) .. (198.45,237.6) .. controls (198.96,229.97) and (208.8,229.74) .. (214.45,226.6) .. controls (232.04,216.83) and (252.2,199.86) .. (244.45,176.6) .. controls (240.5,164.75) and (229.4,158.59) .. (217.45,155.6) .. controls (216.32,155.32) and (211.15,151.88) .. (209.45,151.6) .. controls (202.72,150.48) and (167.75,170.85) .. (166.45,152.6) .. controls (166.14,148.28) and (166.12,143.92) .. (166.45,139.6) .. controls (166.55,138.25) and (169.27,134.78) .. (170.45,133.6) .. controls (173.49,130.56) and (174.31,126.74) .. (177.45,123.6) .. controls (186.89,114.16) and (213.22,113.21) .. (224.45,113.6) .. controls (235.98,114) and (251.09,113.06) .. (262.45,113.6) .. controls (276.27,114.26) and (286.5,113.98) .. (299.45,115.6) .. controls (312.99,117.29) and (309.38,129.53) .. (315.45,135.6) .. controls (318.18,138.33) and (321.52,137.96) .. (326.45,139.6) .. controls (329.28,140.54) and (330.15,144.3) .. (332.45,146.6) .. controls (339.24,153.39) and (344.78,157.91) .. (347.45,168.6) .. controls (349.51,176.85) and (339.72,176.06) .. (337.45,180.6) .. controls (332.11,191.27) and (329.44,192.36) .. (330.45,209.6) .. controls (330.58,211.73) and (333.77,212.58) .. (334.45,214.6) .. controls (335.88,218.89) and (336.51,222.85) .. (337.45,226.6) .. controls (338.21,229.64) and (342.51,231.73) .. (343.45,233.6) .. controls (347.82,242.34) and (351.95,255.09) .. (348.45,265.6) .. controls (347.77,267.64) and (345.27,267.15) .. (344.45,269.6) .. controls (342.81,274.53) and (339.38,280.74) .. (337.45,284.6) .. controls (335.47,288.56) and (328.21,288.84) .. (324.45,292.6) .. controls (320.01,297.04) and (314.36,303.69) .. (310.45,307.6) .. controls (306.06,311.99) and (304.13,323.28) .. (308.45,327.6) .. controls (310.43,329.58) and (319.01,330.91) .. (322.45,331.6) .. controls (323.18,331.75) and (322.74,333.36) .. (323.45,333.6) .. controls (330.3,335.88) and (336.79,337.32) .. (344.45,338.6) .. controls (357.13,340.71) and (368.35,350.89) .. (382.45,351.6) .. controls (409.97,352.98) and (432.89,339.16) .. (446.45,325.6) .. controls (451.9,320.15) and (464.9,320.25) .. (468.45,309.6) .. controls (472.4,297.74) and (464.3,290.45) .. (458.45,284.6) .. controls (455.01,281.16) and (454.78,273.93) .. (451.45,270.6) .. controls (444.12,263.27) and (438.99,263.7) .. (440.45,250.6) .. controls (440.63,248.94) and (442.92,248.18) .. (443.45,246.6) .. controls (444.61,243.12) and (443.9,238.71) .. (445.45,235.6) .. controls (445.91,234.67) and (457.18,228.87) .. (459.45,226.6) .. controls (465.52,220.53) and (468.69,205.89) .. (471.45,197.6) .. controls (473.63,191.05) and (480.73,182.15) .. (476.45,173.6) .. controls (468.31,157.31) and (452.45,159.35) .. (438.45,150.6) .. controls (433.32,147.39) and (420.71,137.7) .. (424.45,129.6) .. controls (429.92,117.76) and (438.96,122.2) .. (450.45,121.6) .. controls (460.53,121.07) and (473.83,121.01) .. (484.45,121.6) .. controls (487.1,121.75) and (491.96,124.86) .. (493.45,125.6) .. controls (495.58,126.67) and (501.27,128.6) .. (500.45,128.6) ;
\draw  [color={rgb, 255:red, 245; green, 166; blue, 35 }  ,draw opacity=1 ][fill={rgb, 255:red, 0; green, 0; blue, 0 }  ,fill opacity=1 ][line width=1.5]  (73.44,141.59) .. controls (73.46,143.57) and (75.08,145.15) .. (77.05,145.12) .. controls (79.03,145.1) and (80.61,143.48) .. (80.58,141.51) .. controls (80.56,139.54) and (78.94,137.96) .. (76.97,137.98) .. controls (75,138) and (73.42,139.62) .. (73.44,141.59) -- cycle ;
\draw  [color={rgb, 255:red, 208; green, 2; blue, 27 }  ,draw opacity=1 ] (76.6,379.95) .. controls (76.6,384.62) and (78.93,386.95) .. (83.6,386.95) -- (271.1,386.95) .. controls (277.77,386.95) and (281.1,389.28) .. (281.1,393.95) .. controls (281.1,389.28) and (284.43,386.95) .. (291.1,386.95)(288.1,386.95) -- (478.6,386.95) .. controls (483.27,386.95) and (485.6,384.62) .. (485.6,379.95) ;
\draw  [color={rgb, 255:red, 74; green, 144; blue, 226 }  ,draw opacity=1 ] (538.6,374.95) .. controls (543.27,374.97) and (545.61,372.65) .. (545.63,367.98) -- (546.07,252.92) .. controls (546.1,246.25) and (548.44,242.93) .. (553.11,242.95) .. controls (548.44,242.93) and (546.12,239.59) .. (546.15,232.92)(546.13,235.92) -- (546.57,121.98) .. controls (546.59,117.31) and (544.27,114.97) .. (539.6,114.95) ;
\draw  [color={rgb, 255:red, 189; green, 16; blue, 224 }  ,draw opacity=1 ][fill={rgb, 255:red, 0; green, 0; blue, 0 }  ,fill opacity=1 ][line width=1.5]  (201.44,150.59) .. controls (201.46,152.57) and (203.08,154.15) .. (205.05,154.12) .. controls (207.03,154.1) and (208.61,152.48) .. (208.58,150.51) .. controls (208.56,148.54) and (206.94,146.96) .. (204.97,146.98) .. controls (203,147) and (201.42,148.62) .. (201.44,150.59) -- cycle ;
\draw  [color={rgb, 255:red, 126; green, 211; blue, 33 }  ,draw opacity=1 ][fill={rgb, 255:red, 0; green, 0; blue, 0 }  ,fill opacity=1 ][line width=1.5]  (308.44,131.59) .. controls (308.46,133.57) and (310.08,135.15) .. (312.05,135.12) .. controls (314.03,135.1) and (315.61,133.48) .. (315.58,131.51) .. controls (315.56,129.54) and (313.94,127.96) .. (311.97,127.98) .. controls (310,128) and (308.42,129.62) .. (308.44,131.59) -- cycle ;
\draw  [color={rgb, 255:red, 155; green, 155; blue, 155 }  ,draw opacity=1 ][fill={rgb, 255:red, 0; green, 0; blue, 0 }  ,fill opacity=1 ][line width=1.5]  (448.44,156.59) .. controls (448.46,158.57) and (450.08,160.15) .. (452.05,160.12) .. controls (454.03,160.1) and (455.61,158.48) .. (455.58,156.51) .. controls (455.56,154.54) and (453.94,152.96) .. (451.97,152.98) .. controls (450,153) and (448.42,154.62) .. (448.44,156.59) -- cycle ;
\draw  [color={rgb, 255:red, 245; green, 166; blue, 35 }  ,draw opacity=1 ][fill={rgb, 255:red, 255; green, 255; blue, 255 }  ,fill opacity=1 ][line width=1.5]  (92.44,205.59) .. controls (92.46,207.57) and (94.08,209.15) .. (96.05,209.12) .. controls (98.03,209.1) and (99.61,207.48) .. (99.58,205.51) .. controls (99.56,203.54) and (97.94,201.96) .. (95.97,201.98) .. controls (94,202) and (92.42,203.62) .. (92.44,205.59) -- cycle ;
\draw  [color={rgb, 255:red, 245; green, 166; blue, 35 }  ,draw opacity=1 ][fill={rgb, 255:red, 10; green, 76; blue, 201 }  ,fill opacity=1 ][line width=1.5]  (102.44,268.59) .. controls (102.46,270.57) and (104.08,272.15) .. (106.05,272.12) .. controls (108.03,272.1) and (109.61,270.48) .. (109.58,268.51) .. controls (109.56,266.54) and (107.94,264.96) .. (105.97,264.98) .. controls (104,265) and (102.42,266.62) .. (102.44,268.59) -- cycle ;
\draw  [color={rgb, 255:red, 245; green, 166; blue, 35 }  ,draw opacity=1 ][fill={rgb, 255:red, 248; green, 231; blue, 28 }  ,fill opacity=1 ][line width=1.5]  (82.44,316.59) .. controls (82.46,318.57) and (84.08,320.15) .. (86.05,320.12) .. controls (88.03,320.1) and (89.61,318.48) .. (89.58,316.51) .. controls (89.56,314.54) and (87.94,312.96) .. (85.97,312.98) .. controls (84,313) and (82.42,314.62) .. (82.44,316.59) -- cycle ;
\draw  [color={rgb, 255:red, 189; green, 16; blue, 224 }  ,draw opacity=1 ][fill={rgb, 255:red, 255; green, 255; blue, 255 }  ,fill opacity=1 ][line width=1.5]  (200.44,229.59) .. controls (200.46,231.57) and (202.08,233.15) .. (204.05,233.12) .. controls (206.03,233.1) and (207.61,231.48) .. (207.58,229.51) .. controls (207.56,227.54) and (205.94,225.96) .. (203.97,225.98) .. controls (202,226) and (200.42,227.62) .. (200.44,229.59) -- cycle ;
\draw  [color={rgb, 255:red, 189; green, 16; blue, 224 }  ,draw opacity=1 ][fill={rgb, 255:red, 10; green, 76; blue, 201 }  ,fill opacity=1 ][line width=1.5]  (220.44,278.59) .. controls (220.46,280.57) and (222.08,282.15) .. (224.05,282.12) .. controls (226.03,282.1) and (227.61,280.48) .. (227.58,278.51) .. controls (227.56,276.54) and (225.94,274.96) .. (223.97,274.98) .. controls (222,275) and (220.42,276.62) .. (220.44,278.59) -- cycle ;
\draw  [color={rgb, 255:red, 189; green, 16; blue, 224 }  ,draw opacity=1 ][fill={rgb, 255:red, 248; green, 231; blue, 28 }  ,fill opacity=1 ][line width=1.5]  (200.44,318.59) .. controls (200.46,320.57) and (202.08,322.15) .. (204.05,322.12) .. controls (206.03,322.1) and (207.61,320.48) .. (207.58,318.51) .. controls (207.56,316.54) and (205.94,314.96) .. (203.97,314.98) .. controls (202,315) and (200.42,316.62) .. (200.44,318.59) -- cycle ;
\draw  [color={rgb, 255:red, 126; green, 211; blue, 33 }  ,draw opacity=1 ][fill={rgb, 255:red, 255; green, 255; blue, 255 }  ,fill opacity=1 ][line width=1.5]  (326.44,209.59) .. controls (326.46,211.57) and (328.08,213.15) .. (330.05,213.12) .. controls (332.03,213.1) and (333.61,211.48) .. (333.58,209.51) .. controls (333.56,207.54) and (331.94,205.96) .. (329.97,205.98) .. controls (328,206) and (326.42,207.62) .. (326.44,209.59) -- cycle ;
\draw  [color={rgb, 255:red, 126; green, 211; blue, 33 }  ,draw opacity=1 ][fill={rgb, 255:red, 10; green, 76; blue, 201 }  ,fill opacity=1 ][line width=1.5]  (338.44,272.59) .. controls (338.46,274.57) and (340.08,276.15) .. (342.05,276.12) .. controls (344.03,276.1) and (345.61,274.48) .. (345.58,272.51) .. controls (345.56,270.54) and (343.94,268.96) .. (341.97,268.98) .. controls (340,269) and (338.42,270.62) .. (338.44,272.59) -- cycle ;
\draw  [color={rgb, 255:red, 126; green, 211; blue, 33 }  ,draw opacity=1 ][fill={rgb, 255:red, 248; green, 231; blue, 28 }  ,fill opacity=1 ][line width=1.5]  (319.44,333.59) .. controls (319.46,335.57) and (321.08,337.15) .. (323.05,337.12) .. controls (325.03,337.1) and (326.61,335.48) .. (326.58,333.51) .. controls (326.56,331.54) and (324.94,329.96) .. (322.97,329.98) .. controls (321,330) and (319.42,331.62) .. (319.44,333.59) -- cycle ;
\draw  [color={rgb, 255:red, 155; green, 155; blue, 155 }  ,draw opacity=1 ][fill={rgb, 255:red, 255; green, 255; blue, 255 }  ,fill opacity=1 ][line width=1.5]  (441.44,235.59) .. controls (441.46,237.57) and (443.08,239.15) .. (445.05,239.12) .. controls (447.03,239.1) and (448.61,237.48) .. (448.58,235.51) .. controls (448.56,233.54) and (446.94,231.96) .. (444.97,231.98) .. controls (443,232) and (441.42,233.62) .. (441.44,235.59) -- cycle ;
\draw  [color={rgb, 255:red, 155; green, 155; blue, 155 }  ,draw opacity=1 ][fill={rgb, 255:red, 10; green, 76; blue, 201 }  ,fill opacity=1 ][line width=1.5]  (457.44,290.59) .. controls (457.46,292.57) and (459.08,294.15) .. (461.05,294.12) .. controls (463.03,294.1) and (464.61,292.48) .. (464.58,290.51) .. controls (464.56,288.54) and (462.94,286.96) .. (460.97,286.98) .. controls (459,287) and (457.42,288.62) .. (457.44,290.59) -- cycle ;
\draw  [color={rgb, 255:red, 155; green, 155; blue, 155 }  ,draw opacity=1 ][fill={rgb, 255:red, 248; green, 231; blue, 28 }  ,fill opacity=1 ][line width=1.5]  (441.44,326.59) .. controls (441.46,328.57) and (443.08,330.15) .. (445.05,330.12) .. controls (447.03,330.1) and (448.61,328.48) .. (448.58,326.51) .. controls (448.56,324.54) and (446.94,322.96) .. (444.97,322.98) .. controls (443,323) and (441.42,324.62) .. (441.44,326.59) -- cycle ;
\draw [color={rgb, 255:red, 255; green, 255; blue, 255 }  ,draw opacity=1 ]   (72,127.1) -- (80,127.1) ;
\draw  [color={rgb, 255:red, 255; green, 255; blue, 255 }  ,draw opacity=1 ][line width=0.75] [line join = round][line cap = round] (84,229.23) .. controls (85.33,229.23) and (86.67,229.23) .. (88,229.23) ;
\draw  [color={rgb, 255:red, 255; green, 255; blue, 255 }  ,draw opacity=1 ][line width=0.75] [line join = round][line cap = round] (87,232.23) .. controls (85.67,232.23) and (84.33,232.23) .. (83,232.23) ;
\draw  [color={rgb, 255:red, 255; green, 255; blue, 255 }  ,draw opacity=1 ][line width=0.75] [line join = round][line cap = round] (99,238.23) .. controls (99,238.23) and (99,238.23) .. (99,238.23) ;
\draw  [color={rgb, 255:red, 255; green, 255; blue, 255 }  ,draw opacity=1 ][line width=0.75] [line join = round][line cap = round] (87,229.23) .. controls (86,229.23) and (85,229.23) .. (84,229.23) ;
\draw  [color={rgb, 255:red, 255; green, 255; blue, 255 }  ,draw opacity=1 ][line width=1.5] [line join = round][line cap = round] (9.5,267.07) .. controls (2.76,295) and (-1.5,318.34) .. (-1.5,347.07) ;
\draw  [color={rgb, 255:red, 255; green, 255; blue, 255 }  ,draw opacity=1 ][line width=1.5] [line join = round][line cap = round] (-13.5,195) .. controls (-18.31,199.81) and (-32.5,200.98) .. (-32.5,207) ;
\draw  [color={rgb, 255:red, 255; green, 255; blue, 255 }  ,draw opacity=1 ][line width=0.75] [line join = round][line cap = round] (-49.5,187.58) .. controls (-66.24,187.58) and (-70.91,172.14) .. (-83.5,164.58) .. controls (-87.17,162.38) and (-100.86,159.58) .. (-98.5,159.58) ;
\draw  [color={rgb, 255:red, 255; green, 255; blue, 255 }  ,draw opacity=1 ][line width=0.75] [line join = round][line cap = round] (241.5,37.92) .. controls (249.01,33.22) and (251.7,31.62) .. (260.5,28.92) .. controls (261.28,28.68) and (272.52,27.92) .. (266.5,27.92) ;
\draw  [color={rgb, 255:red, 255; green, 255; blue, 255 }  ,draw opacity=1 ][line width=0.75] [line join = round][line cap = round] (74.5,129.82) .. controls (75.17,129.82) and (75.83,129.82) .. (76.5,129.82) ;
\draw  [color={rgb, 255:red, 255; green, 255; blue, 255 }  ,draw opacity=1 ][line width=0.75] [line join = round][line cap = round] (75.5,129.82) .. controls (75.17,129.82) and (74.83,129.82) .. (74.5,129.82) ;
\draw  [color={rgb, 255:red, 255; green, 255; blue, 255 }  ,draw opacity=1 ][line width=0.75] [line join = round][line cap = round] (74.5,129.82) .. controls (74.83,129.82) and (75.17,129.82) .. (75.5,129.82) ;
\draw  [color={rgb, 255:red, 255; green, 255; blue, 255 }  ,draw opacity=1 ][line width=0.75] [line join = round][line cap = round] (74.5,129.82) .. controls (74.5,129.82) and (74.5,129.82) .. (74.5,129.82) ;
\draw  [color={rgb, 255:red, 255; green, 255; blue, 255 }  ,draw opacity=1 ][line width=0.75] [line join = round][line cap = round] (74.5,129.82) .. controls (74.5,129.82) and (74.5,129.82) .. (74.5,129.82) ;
\draw  [color={rgb, 255:red, 255; green, 255; blue, 255 }  ,draw opacity=1 ][line width=0.75] [line join = round][line cap = round] (74.5,129.82) .. controls (74.83,129.82) and (75.17,129.82) .. (75.5,129.82) ;
\draw  [color={rgb, 255:red, 255; green, 255; blue, 255 }  ,draw opacity=1 ][line width=0.75] [line join = round][line cap = round] (73.5,129.82) .. controls (74.17,129.82) and (74.83,129.82) .. (75.5,129.82) ;
\draw  [color={rgb, 255:red, 255; green, 255; blue, 255 }  ,draw opacity=1 ][line width=0.75] [line join = round][line cap = round] (102.5,292.87) .. controls (102.83,292.87) and (103.17,292.87) .. (103.5,292.87) ;
\draw  [color={rgb, 255:red, 255; green, 255; blue, 255 }  ,draw opacity=1 ][line width=0.75] [line join = round][line cap = round] (102.5,292.87) .. controls (105.23,292.87) and (102.23,292.87) .. (99.5,292.87) ;
\draw  [color={rgb, 255:red, 255; green, 255; blue, 255 }  ,draw opacity=1 ][line width=0.75] [line join = round][line cap = round] (102.5,292.87) .. controls (104.83,292.87) and (102.83,292.87) .. (100.5,292.87) ;
\draw  [color={rgb, 255:red, 255; green, 255; blue, 255 }  ,draw opacity=1 ][line width=0.75] [line join = round][line cap = round] (101.5,292.87) .. controls (102.17,292.87) and (104.17,292.87) .. (103.5,292.87) .. controls (100.23,292.87) and (99.23,292.87) .. (102.5,292.87) ;
\draw  [color={rgb, 255:red, 255; green, 255; blue, 255 }  ,draw opacity=1 ][line width=0.75] [line join = round][line cap = round] (101.5,292.87) .. controls (103.39,292.87) and (102.39,292.87) .. (100.5,292.87) ;
\draw  [color={rgb, 255:red, 255; green, 255; blue, 255 }  ,draw opacity=1 ][line width=0.75] [line join = round][line cap = round] (100.5,292.87) .. controls (103.17,292.87) and (103.17,292.87) .. (100.5,292.87) ;
\draw  [color={rgb, 255:red, 255; green, 255; blue, 255 }  ,draw opacity=1 ][line width=0.75] [line join = round][line cap = round] (100.5,292.68) .. controls (101.83,292.68) and (103.17,292.68) .. (104.5,292.68) ;
\draw  [color={rgb, 255:red, 255; green, 255; blue, 255 }  ,draw opacity=1 ][line width=0.75] [line join = round][line cap = round] (100.05,293.52) .. controls (101.02,293.52) and (101.98,293.52) .. (102.95,293.52) ;
\draw  [color={rgb, 255:red, 255; green, 255; blue, 255 }  ,draw opacity=1 ][line width=0.75] [line join = round][line cap = round] (99.93,296.32) .. controls (106.87,296.32) and (100.76,296.32) .. (98.86,296.32) ;
\draw  [color={rgb, 255:red, 255; green, 255; blue, 255 }  ,draw opacity=1 ][line width=0.75] [line join = round][line cap = round] (82.4,331.67) .. controls (83.54,331.47) and (84.67,331.26) .. (85.8,331.05) ;
\draw  [color={rgb, 255:red, 255; green, 255; blue, 255 }  ,draw opacity=1 ][line width=0.75] [line join = round][line cap = round] (82.85,334.66) .. controls (83.96,334.38) and (85.08,334.09) .. (86.19,333.81) ;
\draw  [color={rgb, 255:red, 255; green, 255; blue, 255 }  ,draw opacity=1 ][line width=0.75] [line join = round][line cap = round] (192.51,127.99) .. controls (193.77,127.91) and (195.02,127.84) .. (196.27,127.76) ;
\draw  [color={rgb, 255:red, 255; green, 255; blue, 255 }  ,draw opacity=1 ][line width=0.75] [line join = round][line cap = round] (191.51,130.99) .. controls (192.77,130.91) and (194.02,130.84) .. (195.27,130.76) ;
\draw  [color={rgb, 255:red, 255; green, 255; blue, 255 }  ,draw opacity=1 ][line width=0.75] [line join = round][line cap = round] (208.7,215.69) .. controls (209.83,215.15) and (210.96,214.6) .. (212.09,214.05) ;
\draw  [color={rgb, 255:red, 255; green, 255; blue, 255 }  ,draw opacity=1 ][line width=0.75] [line join = round][line cap = round] (207.43,219.63) .. controls (208.54,219.09) and (209.66,218.55) .. (210.77,218.01) ;
\draw  [color={rgb, 255:red, 255; green, 255; blue, 255 }  ,draw opacity=1 ][line width=0.75] [line join = round][line cap = round] (219.57,291.37) .. controls (220.8,291.17) and (222.04,290.97) .. (223.27,290.77) ;
\draw  [color={rgb, 255:red, 255; green, 255; blue, 255 }  ,draw opacity=1 ][line width=0.75] [line join = round][line cap = round] (217.22,294.65) .. controls (218.45,294.46) and (219.69,294.26) .. (220.92,294.06) ;
\draw  [color={rgb, 255:red, 255; green, 255; blue, 255 }  ,draw opacity=1 ][line width=0.75] [line join = round][line cap = round] (210.22,339.65) .. controls (211.45,339.46) and (212.69,339.26) .. (213.92,339.06) ;
\draw  [color={rgb, 255:red, 255; green, 255; blue, 255 }  ,draw opacity=1 ][line width=0.75] [line join = round][line cap = round] (213.22,341.65) .. controls (214.45,341.46) and (215.69,341.26) .. (216.92,341.06) ;
\draw  [color={rgb, 255:red, 255; green, 255; blue, 255 }  ,draw opacity=1 ][line width=0.75] [line join = round][line cap = round] (324.22,150.65) .. controls (325.45,150.46) and (326.69,150.26) .. (327.92,150.06) ;
\draw  [color={rgb, 255:red, 255; green, 255; blue, 255 }  ,draw opacity=1 ][line width=0.75] [line join = round][line cap = round] (326.85,153.2) .. controls (328.11,153) and (329.36,152.8) .. (330.61,152.6) ;
\draw  [color={rgb, 255:red, 255; green, 255; blue, 255 }  ,draw opacity=1 ][line width=0.75] [line join = round][line cap = round] (321.84,229.07) .. controls (323.1,228.96) and (324.36,228.85) .. (325.62,228.74) ;
\draw  [color={rgb, 255:red, 255; green, 255; blue, 255 }  ,draw opacity=1 ][line width=0.75] [line join = round][line cap = round] (321.84,232.07) .. controls (323.1,231.96) and (324.36,231.85) .. (325.62,231.74) ;
\draw  [color={rgb, 255:red, 255; green, 255; blue, 255 }  ,draw opacity=1 ][line width=0.75] [line join = round][line cap = round] (336.3,292.87) .. controls (337.56,292.7) and (338.81,292.52) .. (340.07,292.34) ;
\draw  [color={rgb, 255:red, 255; green, 255; blue, 255 }  ,draw opacity=1 ][line width=0.75] [line join = round][line cap = round] (334.85,296.17) .. controls (336.1,295.99) and (337.36,295.81) .. (338.61,295.63) ;
\draw  [color={rgb, 255:red, 255; green, 255; blue, 255 }  ,draw opacity=1 ][line width=0.75] [line join = round][line cap = round] (321.3,320.87) .. controls (322.56,320.7) and (323.81,320.52) .. (325.07,320.34) ;
\draw  [color={rgb, 255:red, 255; green, 255; blue, 255 }  ,draw opacity=1 ][line width=0.75] [line join = round][line cap = round] (321.3,323.87) .. controls (322.56,323.7) and (323.81,323.52) .. (325.07,323.34) ;
\draw  [color={rgb, 255:red, 255; green, 255; blue, 255 }  ,draw opacity=1 ][line width=0.75] [line join = round][line cap = round] (433.94,137.02) .. controls (435.06,136.95) and (436.17,136.88) .. (437.29,136.8) ;
\draw  [color={rgb, 255:red, 255; green, 255; blue, 255 }  ,draw opacity=1 ][line width=0.75] [line join = round][line cap = round] (434.84,139.84) .. controls (435.93,139.77) and (437.02,139.7) .. (438.12,139.63) ;
\draw  [color={rgb, 255:red, 255; green, 255; blue, 255 }  ,draw opacity=1 ][line width=0.75] [line join = round][line cap = round] (449.84,213.84) .. controls (450.93,213.77) and (452.02,213.7) .. (453.12,213.63) ;
\draw  [color={rgb, 255:red, 255; green, 255; blue, 255 }  ,draw opacity=1 ][line width=0.75] [line join = round][line cap = round] (447.84,216.84) .. controls (448.93,216.77) and (450.02,216.7) .. (451.12,216.63) ;
\draw  [color={rgb, 255:red, 255; green, 255; blue, 255 }  ,draw opacity=1 ][line width=0.75] [line join = round][line cap = round] (464.84,277.77) .. controls (465.93,277.75) and (467.03,277.72) .. (468.12,277.7) ;
\draw  [color={rgb, 255:red, 255; green, 255; blue, 255 }  ,draw opacity=1 ][line width=0.75] [line join = round][line cap = round] (463.84,280.77) .. controls (464.93,280.75) and (466.03,280.72) .. (467.12,280.7) ;
\draw  [color={rgb, 255:red, 255; green, 255; blue, 255 }  ,draw opacity=1 ][line width=0.75] [line join = round][line cap = round] (462.86,345) .. controls (463.94,344.82) and (465.02,344.65) .. (466.1,344.47) ;
\draw  [color={rgb, 255:red, 255; green, 255; blue, 255 }  ,draw opacity=1 ][line width=0.75] [line join = round][line cap = round] (464.16,347.89) .. controls (465.33,347.67) and (466.5,347.45) .. (467.66,347.22) ;

\draw (219,404) node [anchor=north west][inner sep=0.75pt]  [font=\scriptsize] [align=left] {\begin{minipage}[lt]{88.21pt}\setlength\topsep{0pt}
\begin{center}
$\displaystyle r+1$ vertical fibers
\end{center}
$\displaystyle t =\zeta ^{j}\overline{t} ,\ j=1,\dotsc ,r+1$
\end{minipage}};
\draw (603,155) node [anchor=north west][inner sep=0.75pt]   [align=left] {$ $};
\draw (569,230) node [anchor=north west][inner sep=0.75pt]  [font=\scriptsize] [align=left] {\begin{minipage}[lt]{70.9pt}\setlength\topsep{0pt}
\begin{center}
corresponding\\$\displaystyle r+1$ horizontal fibers
\end{center}

\end{minipage}};

\end{tikzpicture}

}
    \caption{Visual intuition behind the code construction: the picture gives an idea of how to visualise the set $S_{\bar{t}}$ of $(r+1)^2$ \emph{nice} points obtained from a \emph{nice} $\bar{t}$. Each point lies on exactly one vertical fiber $\Pi_t^{-1}(\zeta^j\bar{t})$, $j=0, \ldots,r$, and on exactly one horizontal fiber $\Pi_x^{-1}(\bar{x})$, where $\bar{x}$ is one of the distinct $r+1$ solutions of $\PP_{\bar{t}}(x)=0$. In the picture, the vertical fibers are depicted in red, the horizontal ones in blue and the curve $\MM_r$ as the dotted curve in black. With notations as in \cref{def:evalpoints}, each point $P_{\bar{i},\bar{j}}^{(\bar{\ell})}$ is marked with two colours, that help visualising its two corresponding sets $H_{\bar{i}}^{(\bar{\ell})}$ and $V_{\bar{j}}^{(\bar{\ell})}$. Indeed, for any point $P_{\bar{i},\bar{j}}^{(\bar{\ell})}$, the remaining points with the same contour colour are those in the set $V_{\bar{j}}^{(\bar{\ell})}$, while the remaining points with the same filling colour are those in the set $H_{\bar{i}}^{(\bar{\ell})}$. Note that the intersection of each vertical and horizontal fiber consists only of $2$ points. This is represented in the picture by leaving some white space around the third point where, in the two-dimensional drawing, each vertical and each horizontal fiber appear to be overlapping.}
    \label{fig:visualintuition}
\end{figure}
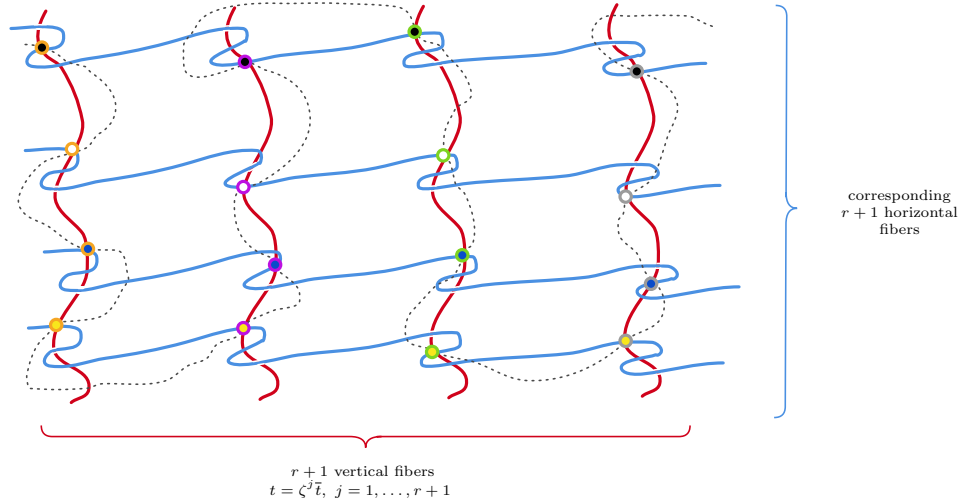

\begin{definition}
    \label[definition]{def:code}
    Let $r\in \mathbb{Z}_{\geq 3}$ odd and $q \equiv 1\pmod{r+1}$. Let $\E_r$ as in \cref{eq:Erdef} and $\M_r$ as in \cref{eq:Mrdef} be defined over $\Fqm$, where $m$ is a positive integer such that $\mathcal{G}_m \neq \emptyset$.
    Choose subsets $\mathcal{G}_m^{(1)},\ldots, \mathcal{G}_m^{(b)}$ of $\mathcal{G}_m$, where $|\mathcal{G}_m|=(r+1)M$ and $b\in \mathbb{Z}$, $1\leq b \leq M$. Let $B:=\cup_{\ell=1}^b \mathcal{G}_m^{(\ell)}$ and $S_B$ as in \cref{def:evalpoints}.

    Furthermore, consider the $\Fqm$-vector space
    \begin{align}
    \label{eq:L:definition}
    L:= \langle x^it^j, x^{r-1}t^h\mid i=1,\ldots, r-2, \  j=0,\ldots, r-1, \ h=0,\ldots, r-2 \rangle,
    \end{align}
    and the linear map evaluating the functions in $L$ at the points in $S_B$ 
    \begin{align*}
        \mathrm{ev}_{S_B}: \ &L \longrightarrow \left(\Fqm\right)^n\\
        &f \longmapsto \left(f(P_{1,1}^{(1)}), \ldots, f(P_{r+1,r+1}^{(b)})\right). 
    \end{align*}
    Note that the functions in $L$ have no poles at the points in $S_B$.
    For more compact notations, we also write $\mathrm{ev}_{S_B}(f)=\left(f(P_{i,j}^{(\ell)})\right)_{i,j=1,\ldots,r+1}^{\ell=1,\ldots,b}$.
    We define an AG code $\mathfrak{C}(B,L)$ of length $n:=|S_B|=b(r+1)^2$ on $\E_r$ as the image of $\mathrm{ev}_{S_B}$, that is,
\begin{equation}
\label{eq:codedef:gen}
    \mathfrak{C}(B,L):=\left\{ \left(f(P_{i,j}^{(\ell)})\right)_{i,j=1,\ldots,r+1}^{\ell=1,\ldots,b} \ \biggm| \ f\in L\right\}.
\end{equation}
\end{definition}

\begin{remark}
    \label[remark]{rem:L:definition}
    The elements in the vector space $L$ defined in \cref{eq:L:definition} are polynomials in $x$ and $t$ with coefficients in $\Fqm$ and such that, for any $f\in L\setminus\{0\}$, both $\mathrm{deg}_x \ f\leq r-1$ and $\mathrm{deg}_t \ f\leq r-1$. More precisely, any element $f(x,t)\in L$ can be written as
    \begin{equation*}
        f(x,t) =  \sum_{i=1}^{r-2}\sum_{j=0}^{r-1}a_{ij}x^it^j + x^{r-1}\sum_{h=0}^{r-2}b_{h}t^h,
    \end{equation*}
    $a_{ij}, b_h\in \Fqm$.
    Moreover, from \cref{eq:L:definition}, it is also clear that the dimension of $L$ as an $\Fqm$-vector space is $(r-2)r + (r-1) = r(r-1) - 1$.
\end{remark}

We now investigate the dimension, minimum distance, locality and availability of a code $\CC(B,L)$.

\section{Parameters of the codes}
\label{sec:parameters}

In this section, we investigate the dimension, minimum distance, locality and availability of a code $\CC(B,L)$ as in \cref{def:code}. Throughout the section, for less cumbersome notations, we simply denote the set of evaluation points $S_B$ of $\CC(B,L)$ by $S$. 

\subsection{Dimension, locality and availability}

We start by computing the dimension of $\CC(B,L)$.

\begin{lemma}
\label[lemma]{lemma:genconstruction:dim}
    A code $\mathfrak{C}(B,L)$ defined in \cref{def:code} has dimension $k=r(r-1) - 1$. 
\end{lemma}
\begin{proof}
To prove the lemma, it suffices to show that the map $\mathrm{ev}_S$ is injective, so that $k:=\mathrm{dim} \ \mathfrak{C}= \mathrm{dim}_{\Fqm} \ L$, since, as observed in \cref{rem:L:definition}, $\mathrm{dim}_{\Fqm} \ L=r(r-1) - 1$.

We start by noting that any element $f \in L$ can be written as $f=\sum_{s=1}^{r-1}a_s(t)x^s$, where $a_s(t)\in \Fqm[t]$ is a polynomial of degree $\mathrm{deg} \ a_s(t)\leq r-1$ for all $s=1,\ldots, r-1$.
Let now $f=\sum_{s=1}^{r-1}a_s(t)x^s\in \mathrm{ker}(ev_S)$, and observe that $ev_S(f)=(0,\ldots,0)$ implies in particular that, for any $\ell=1,\ldots, b$ and any $j=1,\dots, r+1$, $ev_S|_{V_j^{(\ell)}}(f)=(0,\ldots,0)$. Now, by assumption, $f(x,\bar{t}_j^{(\ell)})$ is a univariate polynomial of degree at most $r-1$ having $r+1$ distinct zeros, namely $\bar{x}_i^{(\ell)}$, for $i=1,\ldots, r+1$. This implies that $a_s(\bar{t}_j^{(\ell)}) = 0$ for all $s=1,\ldots, r+1$. Since, for all $s=1,\ldots, r+1$ this happens for all $j=1,\ldots, r+1$, this implies in particular that $a_s(t) \equiv 0$ as a polynomial in $\Fqm[t]$, for all $s=1,\ldots,r+1$. This shows that $f=\sum_{s=1}^{r-1}a_s(t)x^s\equiv 0$ as a polynomial in $\Fqm[x,t]$, and hence the evaluation map $ev_S$ is injective. The dimension $k$ of the code is therefore equal to the dimension of $L$ as a vector space over $\Fqm$, that is $k=r(r-1) - 1$.
\end{proof}

The following lemma, concerning the locality and availability of a code $\CC(B,L)$, is inspired by (and can be seen as the natural counterpart of) \cite[Lemma 3.4]{AAAOAVA24}.

\begin{lemma}
\label[lemma]{lemma:genconstruction:recovery}
    A code $\mathfrak{C}(B,L)$ as in \cref{def:code} has locality $r$ and availability $2$. More precisely, for any $f\in L\setminus\{0\}$, each symbol $f(P_{\bar{i},\bar{j}}^{(\bar{\ell})})$ of the codeword $\mathfrak{c} := \left(f(P_{i,j}^{(\ell)})\right)_{i,j=1,\ldots,r+1}^{\ell=1,\ldots,b}$ has two distinct recovery sets, namely
    \begin{align*}
        f(H_{\bar{i}}^{(\bar{\ell})})&:=\left\{f(P_{\bar{i},j}^{(\bar{\ell})}) \mid j=1,\ldots,r+1 \ \mbox{and} \ j\neq \bar{j}\right\},\\
        f(V_{\bar{j}}^{(\bar{\ell})})&:=\left\{f(P_{i,\bar{j}}^{(\bar{\ell})}) \mid i=1,\ldots,r+1 \ \mbox{and} \  i\neq \bar{i}\right\}.
    \end{align*}
    We refer to $f(H_{\bar{i}}^{(\bar{\ell})})$ (resp. $f(V_{\bar{j}}^{(\bar{\ell})})$) as a \emph{horizontal} (resp. \emph{vertical}) recovery set for $f(P_{\bar{i},\bar{j}}^{(\bar{\ell})})$ (see \cref{eq:horizontalset} and \cref{eq:verticalset}).
\end{lemma}
\begin{proof}
Consider $f=\sum_{s=1}^{r-1}a_s(t)x^s \in L\setminus\{0\}$, where $a_s(t)\in \Fqm[t]$ is a polynomial of degree $\mathrm{deg} \ a_s(t)\leq r-1$ for all $s=1,\ldots, r-1$, and the codeword $\mathfrak{c}$ obtained by evaluating $f$ at the points in $S$, that is,  
\begin{equation*}
    \mathfrak{c} := \left(f(P_{i,j}^{(\ell)})\right)_{i,j=1,\ldots,r+1}^{\ell=1,\ldots,b}.
\end{equation*}
If $\mathfrak{c}$ has a missing symbol, which without loss of generality we can assume to be $f(P_{1,1}^{(1)})$, then we proceed as follows. 

Suppose first that we wish to recover $f(P_{1,1}^{(1)})$ using the vertical recovery set $f(V_1^{(1)})$. Since the points in $V_1^{(1)}$ all lie on the same vertical fiber $\Pi_t^{-1}(\bar{t}_1^{(1)})$, for less cumbersome notations, throughout the rest of the first part of the proof we simply write $t=\bar{t}_1^{(1)}$ and $x_i = \bar{x}_i^{(1)}$. 
Then, to recover $f(P_{1,1}^{(1)})$ from the values $f(P_{i,1}^{(1)})$, $i=2,\ldots, r+1$, (that is, from the values in $f(V_1^{(1)})$) and from the coordinates of the points in $V_1^{(1)}$, 
we can use Lagrange interpolation to retrieve the coefficients $a_1(t), \ldots, a_{r-1}(t) \in \Fqm$ of the polynomial $f=\sum_{s=1}^{r-1}a_s(t)x^s.$
Indeed, by assumption we know the $r$ values $f(P_{i,1}^{(1)})$, $i=2,\ldots, r+1$, from which the coefficients $a_1(t), \ldots, a_{r-1}(t)$ can be interpolated. The missing symbol $f(P_{1,1}^{(1)})$ can hence be computed by
\begin{equation*}
    f(P_{1,1}^{(1)})=\sum_{s=1}^{r-1}a_s(t)x_1^s.
\end{equation*}

Suppose now that we wish to recover $f(P_{1,1}^{(1)})$ using instead the horizontal recovery set $f(H_1^{(1)})$. We start by observing that the function $f=\sum_{s=1}^{r-1}a_s(t)x^s$ can be rewritten as $f=\sum_{s=0}^{r-1}d_s(x)t^s$, for $d_s(x)\in \Fqm[x]$ with $\mathrm{deg}\ d_s(x)\leq r-1$. 
Now, mutatis mutandis, the recovery strategy is entirely similar as above, since the coefficients of the polynomial $f=\sum_{s=0}^{r-1}d_s(x)t^s$ can be interpolated from the $r$ values $f(P_{1,j}^{(1)})$, $j=2,\ldots, r+1$. 

\end{proof}

\subsection{Estimates on the minimum distance}
\label{subsec:d:bound}

In this subsection, after discussing a general known upper bound for the minimum distance of the codes $\CC(B,L)$, we derive a lower bound for the minimum distance of these codes, which is computed by studying the valuation of the functions in $L$ at certain places of $\Fqm(\MM_r)$. 

From \cite[Eq.~(16)]{TBF16}, we have the following Singleton-type upper bound for the minimum distance $d$ of an LRC with length $n$, dimension $k$, locality $r$ and availability $2$:
    \begin{equation}
    \label{eq:tbf:upper}
        d \leq n - \left(k - 1 + \left\lfloor\frac{k-1}{r}\right\rfloor + \left\lfloor\frac{k-1}{r^2}\right\rfloor \right).
    \end{equation}
For a code $\CC(B,L)$, which has length $n=b(r+1)^2$ (see \cref{def:code}) and dimension $k=r(r-1)-1$ by \cref{lemma:genconstruction:dim}, this bound reads as
     \begin{equation*}
        d \leq b(r+1)^2 - \left(r^2 - r - 2 + \left\lfloor\frac{r(r-1)-2}{r}\right\rfloor + \left\lfloor\frac{r(r-1)-2}{r^2}\right\rfloor \right),
    \end{equation*}
hence we obtain
    \begin{align}
    \label{eq:tbf:upper:explicit}
    \nonumber
        d &\leq b(r+1)^2 - \left(r^2 - r - 2 + \left\lfloor\frac{r(r-1)-2}{r}\right\rfloor + \left\lfloor\frac{r(r-1)-2}{r^2}\right\rfloor \right)\\
        &= b(r+1)^2 - \left(r^2 - r - 2 + r - 2 \right)\\ \nonumber
        &= b(r+1)^2 - (r^2 - 4).
    \end{align}

Throughout the rest of this section, we investigate a lower bound for $d$. The underlying idea is to look at the valuations of the functions in $L$ at places of $\Fqm(\MM_r)$ corresponding to points of $\MM_r$ that are not evaluation points of $\CC(B,L)$. More precisely, we study the valuations of the elements of $L$ at the places of $\Fqm(\MM_r)$ lying over the pole of $t$. In \cref{prop:bound:d}, we find a lower bound on $d$ that turns out to be sharp when $r=3$, except when $b=1$. However, when $b=1$, we are able to actually compute the minimum distance in all cases, see \cref{prop:mindist:beq1}. Explicit examples for $r=3$ are collected in \cref{table:codes}.

We start by studying the splitting behaviour of $(t=\infty)$ in $\Fqm(x,t)/\Fqm(t)$. To this aim, we determine the Newton arc of $\PP_t(T)$ with respect to $(t=\infty)$. 

From observations in \cref{sec:code:construction}, we already know that $[\Fqm(t,x):\Fqm(t)]=r+1$ and that the extension $\Fqm(x,t)/\Fqm(t)$ is defined by $\PP_t(x)=0$, where $\PP_t(T)$ (see \cref{eq:PPt}) is an irreducible polynomial of $\Fqm(t)[T]$. Let now $\mathcal{O}_{(t=\infty)}$ be the DVR associated to the place $(t=\infty)$ in $\Fqm(t)$ and consider
\begin{equation*}
	\tilde{\PP_t}(T):=\frac{1}{t^{r+1}}\PP_t(T) = \frac{1}{t^{r+1}}T^{r+1} + \frac{2}{t^{r+1}}T^{\frac{r+1}{2}} - \frac{1}{t^{r+1}}T^3 + \frac{t^{r+1} + 1}{t^{r+1}}T^2 - T + \frac{1}{t^{r+1}} \in \mathcal{O}_{(t=\infty)}[T].
\end{equation*}  
Modulo a vertical translation, the Newton arc of $\PP_t(T)$ with respect to $(t=\infty)$ coincides with that of $\tilde{\PP_t}(T)$, meaning that the slopes of the segments that form the two arcs are the same (see \cref{subsec:newton:background}). Therefore, determining the Newton arc of $\tilde{\PP_t}(T)$ also yields the Newton arc of $\PP_t(T)$ and is enough to deduce the splitting behaviour of $(t=\infty)$ in $\Fqm(t,x)/\Fqm(t)$.

\begin{lemma}
\label[lemma]{lemma:newtonarc}
    For any $r\in \mathbb{Z}_{\geq 3}$, the Newton arc of $\tilde{\PP_t}(T)$ with respect to $(t=\infty)$ consists of three line segments $\mathcal{L}_1$, $\mathcal{L}_2$ and $\mathcal{L}_3$ with slopes $\rho(\mathcal{L}_1)= -(r+1)/1$, $\rho(\mathcal{L}_2)= -0/1$ and $\rho(\mathcal{L}_3)= ((r+1)/2) / ((r-1)/2)$. Furthermore, each segment has the following associated polynomials:
    \begin{align}
    \label{eq:gamma:delta:pols}
    \nonumber
    \gamma_{\mathcal{L}_1}(T) &= - T + 1, \quad &&\delta_{\mathcal{L}_1}(T)   = T - 1,\\
    \gamma_{\mathcal{L}_2}(T)  &= - T + 1, \quad &&\delta_{\mathcal{L}_2}(T)   = T - 1,\\\nonumber
    \gamma_{\mathcal{L}_3}(T)  &= T^{(r+1)-2} + 1 = T^{r-1} + 1 = (T^{(r-1)/2})^2 + 1, \quad &&\delta_{\mathcal{L}_3}(T)  = T^2 + 1.
\end{align}
\end{lemma}
\begin{proof}
    The lemma follows simply by observing that the support set (see \cref{def:supportset}) $\mathcal{S}_{\tilde{\PP_t}(T)}$ of $\tilde{\PP_t}(T)$ at $(t=\infty)$ is, for $r\geq 5$,
    \begin{equation*}
	\mathcal{S}_{\tilde{\PP_t}(T)} = \{(0,r+1), (1,0), (2,0), (3,r+1), ((r+1)/2,r+1), (r+1,r+1)\},
    \end{equation*}
    while for $r=3$ is
    \begin{equation*}
	\mathcal{S}_{\tilde{\PP_t}(T)} = \{(0,r+1), (1,0), (2,0), (3,r+1), (r+1,r+1)\}.
    \end{equation*}
    In both cases, the shape of the Newton arc is hence the same. Note that the way in which we write the slopes in the statement is instrumental to make explicit their expression as fractions $-a/b$, where $a,b\in \mathbb{Z}$, $b>0$, $\mathrm{gcd} \ (a,b) = 1$, so that for all $i=1,2,3$ it is immediate to obtain the polynomials $\gamma_{\mathcal{L}_i}(T)$ and $\delta_{\mathcal{L}_i}(T)$, as defined in \cref{subsec:newton:background}.
\end{proof}

\begin{center}
\begin{figure}[ht]
\adjustbox{scale=0.7}{%
\tikzset{every picture/.style={line width=0.75pt}} 
\begin{tikzpicture}[x=0.75pt,y=0.75pt,yscale=-1,xscale=1]

\draw [color={rgb, 255:red, 9; green, 96; blue, 191 }  ,draw opacity=1 ][line width=1.5]  (43.75,177.98) -- (386.75,177.98)(59.63,-98.5) -- (59.63,204) (379.75,172.98) -- (386.75,177.98) -- (379.75,182.98) (54.63,-91.5) -- (59.63,-98.5) -- (64.63,-91.5) (90.63,172.98) -- (90.63,182.98)(121.63,172.98) -- (121.63,182.98)(152.63,172.98) -- (152.63,182.98)(183.63,172.98) -- (183.63,182.98)(214.63,172.98) -- (214.63,182.98)(245.63,172.98) -- (245.63,182.98)(276.63,172.98) -- (276.63,182.98)(307.63,172.98) -- (307.63,182.98)(338.63,172.98) -- (338.63,182.98)(369.63,172.98) -- (369.63,182.98)(54.63,146.98) -- (64.63,146.98)(54.63,115.98) -- (64.63,115.98)(54.63,84.98) -- (64.63,84.98)(54.63,53.98) -- (64.63,53.98)(54.63,22.98) -- (64.63,22.98)(54.63,-8.02) -- (64.63,-8.02)(54.63,-39.02) -- (64.63,-39.02)(54.63,-70.02) -- (64.63,-70.02) ;
\draw   ;
\draw [color={rgb, 255:red, 208; green, 2; blue, 27 }  ,draw opacity=1 ]   (60.1,-70.5) -- (91.1,177.5) ;
\draw [color={rgb, 255:red, 208; green, 2; blue, 27 }  ,draw opacity=1 ]   (91.1,177.5) -- (121.4,178) ;
\draw [color={rgb, 255:red, 208; green, 2; blue, 27 }  ,draw opacity=1 ]   (121.4,178) -- (307,-68) ;
\draw  [color={rgb, 255:red, 208; green, 2; blue, 27 }  ,draw opacity=1 ][fill={rgb, 255:red, 208; green, 2; blue, 27 }  ,fill opacity=1 ] (153.22,-72.19) .. controls (153.23,-71.64) and (153.69,-71.2) .. (154.24,-71.22) .. controls (154.8,-71.23) and (155.23,-71.69) .. (155.22,-72.24) .. controls (155.2,-72.79) and (154.74,-73.23) .. (154.19,-73.21) .. controls (153.64,-73.2) and (153.2,-72.74) .. (153.22,-72.19) -- cycle ;
\draw  [color={rgb, 255:red, 208; green, 2; blue, 27 }  ,draw opacity=1 ][fill={rgb, 255:red, 208; green, 2; blue, 27 }  ,fill opacity=1 ] (213.22,-72) .. controls (213.23,-71.45) and (213.69,-71.01) .. (214.24,-71.03) .. controls (214.8,-71.04) and (215.23,-71.5) .. (215.22,-72.06) .. controls (215.2,-72.61) and (214.74,-73.04) .. (214.19,-73.03) .. controls (213.64,-73.01) and (213.2,-72.55) .. (213.22,-72) -- cycle ;

\draw (50.17,188.4) node [anchor=north west][inner sep=0.75pt]  [font=\scriptsize,xslant=0.02]  {$0$};
\draw (118.17,192.4) node [anchor=north west][inner sep=0.75pt]  [font=\scriptsize,xslant=0.02]  {$2$};
\draw (87.17,192.4) node [anchor=north west][inner sep=0.75pt]  [font=\scriptsize,xslant=0.02]  {$1$};
\draw (149.17,192.4) node [anchor=north west][inner sep=0.75pt]  [font=\scriptsize,xslant=0.02]  {$3$};
\draw (202.17,191.4) node [anchor=north west][inner sep=0.75pt]  [font=\scriptsize,xslant=0.02]  {$\frac{r+1}{2}$};
\draw (296.17,193.4) node [anchor=north west][inner sep=0.75pt]  [font=\scriptsize,xslant=0.02]  {$r+1$};
\draw (26.17,-74.6) node [anchor=north west][inner sep=0.75pt]  [font=\scriptsize,xslant=0.02]  {$r+1$};

\end{tikzpicture}
}
\caption{Newton arc of $\tilde{\PP_t}(T)$ with respect to $(t=\infty)$. \label{fig:newtonarc}}
\end{figure}
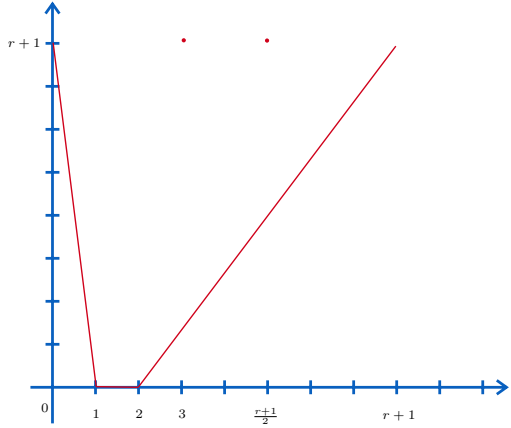
\end{center}

By \cref{thm:newtonarc:splitting} and \cref{cor:newtonarc:splitting}, there are only two possibilities for the splitting behaviour of $(t=\infty)$ in $\Fqm(t,x)/\Fqm(t)$, which are described in the following lemma.
\begin{lemma}
\label[lemma]{lemma:splitting}
    Let notations be as in \cref{lemma:newtonarc}. Then:
    \begin{enumerate}
        \item If $-1$ is a square in $\Fqm$, the place $(t=\infty)$ has four places lying over it in $\Fqm(t,x)/\Fqm(t)$, namely two places $P_1, P_2$ with $e(P_i|(t=\infty))=1$ and $f(P_i|(t=\infty))=1$, for $i=1,2$, and two places $P_3, P_4$ with $e(P_i|(t=\infty))=(r-1)/2$ and $f(P_i|(t=\infty))=1$, for $i=3,4$.
        \item If $-1$ is not a square in $\Fqm$, the place $(t=\infty)$ has three places lying over it in $\Fqm(t,x)/\Fqm(t)$, namely two places $P_1, P_2$ with $e(P_i|(t=\infty))=1$ and $f(P_i|(t=\infty))=1$, for $i=1,2$, and one place $P_3$ with $e(P_3|(t=\infty))=(r-1)/2$ and $f(P_3|(t=\infty))=2$.
    \end{enumerate}
\end{lemma}
\begin{proof}
    The proof follows directly from \cref{thm:newtonarc:splitting} and \cref{cor:newtonarc:splitting}, simply by observing that if $-1$ is a square in $\Fqm$ then the polynomial $\delta_{\mathcal{L}_3}(T)=T^2+1$ splits over $\Fqm$, while it is irreducible in $\Fqm$ if $-1$ is not a square. In case $(1)$, the place $P_1$ corresponds to the line segment $\mathcal{L}_1$, the place $P_2$ corresponds to $\mathcal{L}_2$ and $P_3, P_4$ correspond to $\mathcal{L}_3$. In case $(2)$, similarly, the place $P_1$ corresponds to the line segment $\mathcal{L}_1$, the place $P_2$ corresponds to $\mathcal{L}_2$ and $P_3$ corresponds to $\mathcal{L}_3$.
\end{proof}

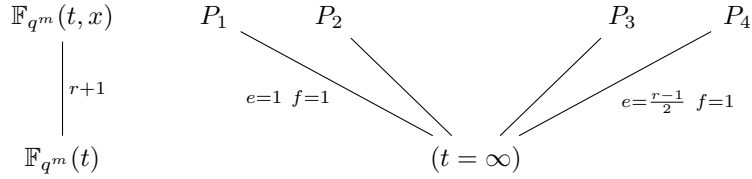
\begin{figure}[ht]
    \centering
\adjustbox{scale=1,center}{
\begin{tikzcd}
{\mathbb{F}_{q^m}(t,x)} \arrow[dd, "r+1", no head] & P_1 \arrow[rrdd, "{e=1 \ f=1}"', no head] & P_2 \arrow[rdd, no head] &                                                                                 & P_3 & P_4 \\
                                               &                                         &                          &                                                                                 &     &     \\
\mathbb{F}_{q^m}(t)                                &                                         &                          & (t=\infty) \arrow[rruu, "{e=\frac{r-1}{2} \ f=1}"', no head] \arrow[ruu, no head] &     &    
\end{tikzcd}
}
\caption{Splitting of $(t=\infty)$ in case $(1)$ of \cref{lemma:splitting} \label{fig:splitting1}}
\end{figure}

\begin{figure}[ht]
    \centering
\adjustbox{scale=1,center}{%
\begin{tikzcd}
{\Fqm(t,x)} \arrow[dd, "r+1", no head] & P_1 \arrow[rdd, "e=1 \ f=1"', no head] & P_2 \arrow[dd, "e=1 \ f=1" description, no head] & {P_3} \arrow[ldd, "e=\frac{r-1}{2} \ f=2", no head] \\
                                               &                                        &                                                  &                                                         \\
\Fqm(t)                                &                                        & (t=\infty)                                       &                                                        
\end{tikzcd}
}
    \caption{Splitting of $(t=\infty)$ in case $(2)$ of \cref{lemma:splitting} \label{fig:splitting2}}
\end{figure}
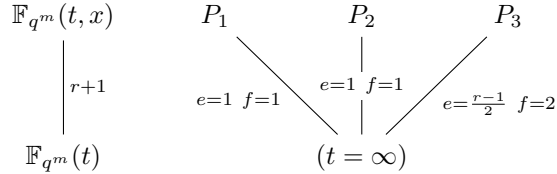

\begin{remark}
    \label[remark]{rem:-1square}
 By the law of quadratic reciprocity, $-1$ is a square modulo $q^m$ if and only if $q^m\equiv 1  \pmod{4}$. If $r=3$, since we are assuming $q\equiv 1 \pmod{r+1}$, we are hence always in case $(1)$ of \cref{lemma:splitting}. 
\end{remark}

We now compute the valuations of the generators of the vector space $L$ (see \cref{def:code}) at the places of $\Fqm(t,x)$ lying over $(t=\infty)$, both in cases $(1)$ and $(2)$ of \cref{lemma:splitting}. Observe that, to this aim, it is enough to compute the valuation of $x$ at $P_1,P_2,P_3,P_4$ (resp. $P_1,P_2,P_3$), since each generator of $L$ is a monomial involving powers of $x$ and $t$, and we already know the valuation of $t$ at $P_1,P_2,P_3,P_4$ (resp. $P_1,P_2,P_3$) from \cref{lemma:splitting}.

\begin{lemma}
    \label[lemma]{lemma:valuations:L:x}
    Let notations be as in \cref{lemma:newtonarc} and \cref{lemma:splitting}. Moreover, as before let $\Fqm(\MM_r)$ be the function field of $\MM_r$, described as $\Fqm(t,x)$ in \cref{rem:functionfield:notation}.
     
     \begin{enumerate}
        \item Suppose that $-1$ is a square in $\Fqm$.
        Then 
        \begin{align}
        \label{eq:valuations:(1):x}
            &v_{P_1}(x)= r+1,\quad &&v_{P_2}(x) = 0, \quad &&&v_{P_3}(x)= v_{P_4}(x) = - \frac{r+1}{2}.
        \end{align}
        \item Suppose that $-1$ is not a square in $\Fqm$. 
        Then
        \begin{align}
        \label{eq:valuations:(2):x}
            &v_{P_1}(x)= r+1,\quad &&v_{P_2}(x) = 0, \quad &&&v_{P_3}(x)= - \frac{r+1}{2}.
        \end{align}
    \end{enumerate}
\end{lemma}
\begin{proof}
    Both \cref{eq:valuations:(1):x} and \cref{eq:valuations:(2):x} follow directly from \cref{thm:newtonarc:splitting}, \cref{cor:newtonarc:splitting} and \cref{lemma:newtonarc}. 
\end{proof}

\begin{remark}
    \label[remark]{rem:valuations:L:generators}
    By \cref{lemma:valuations:L:x} and \cref{lemma:splitting}, we can determine the valuation at $P_1,P_2,P_3,P_4$ (resp. $P_1,P_2,P_3$) of all the generators of $L$. In case $(1)$ of \cref{lemma:splitting}, we have
    \begin{align}
    \label{eq:val:generators:1}
    \nonumber
        v_{P_1}(x^it^j) &= (r+1)i - j,\\
        v_{P_2}(x^it^j) &= -j,\\ \nonumber
        v_{P_3}(x^it^j) &= v_{P_4}(x^it^j) = - \frac{r+1}{2} \ i -  \frac{r-1}{2} \ j,
    \end{align}
    for $i = 1,\ldots, r-2, \ j = 0, \ldots, r-1$, while
    \begin{align}
    \label{eq:val:generators:2}
    \nonumber
        v_{P_1}(x^{r-1}t^h) &= (r+1)(r-1) - h,\\
        v_{P_2}(x^{r-1}t^h) &= -h,\\ \nonumber
        v_{P_3}(x^{r-1}t^h) &= v_{P_4}(x^{r-1}t^h) = - \frac{r^2-1}{2} -  \frac{r-1}{2} \ h,
    \end{align}
    for $h = 0,\ldots, r-2$, and, similarly, in case $(2)$ of \cref{lemma:splitting} we have
    \begin{align}
    \label{eq:val:generators:3}
    \nonumber
        v_{P_1}(x^it^j) &= (r+1)i - j,\\
        v_{P_2}(x^it^j) &= -j,\\ \nonumber
        v_{P_3}(x^it^j) & = - \frac{r+1}{2} \ i -  \frac{r-1}{2} \ j,
    \end{align}
    for $i = 1,\ldots, r-2, \ j = 0, \ldots, r-1$, while 
    \begin{align}
    \label{eq:val:generators:4}
    \nonumber
        v_{P_1}(x^{r-1}t^h) &= (r+1)(r-1) - h,\\
        v_{P_2}(x^{r-1}t^h) &= -h,\\ \nonumber
        v_{P_3}(x^{r-1}t^h) &= - \frac{(r-1)^2}{2} -  \frac{r-1}{2} \ h,
    \end{align}
    for $h = 0,\ldots, r-2$.
    
    From \cref{eq:val:generators:1}, \cref{eq:val:generators:2}, \cref{eq:val:generators:3} and \cref{eq:val:generators:4}, it follows that, both in case $(1)$ and $(2)$ of \cref{lemma:splitting}, the place $P_1$ is a common zero of all the generators of $L$, and hence it is a common zero of all the functions in $L$. More precisely, we have that 
    \begin{equation}
    \label{eq:minval:P1}
        \mathrm{min}_{f \in L\setminus\{0\}}\{v_{P_1}(f)\} = (r+1) - (r-1) = 2.
    \end{equation}
    On the other hand, observe that the valuations of the generators of $L$ at all the other places $P_2,P_3,P_4$ (resp. $P_2,P_3$) are negative.
\end{remark}

\begin{remark}
    \label[remark]{rem:poledivs}
    As observed in the beginning of \cref{sec:code:construction}, $\Fqm(t,x)/\Fqm(x)$ is a Kummer extension of degree $r+1$. By \cref{lemma:valuations:L:x}, if $-1$ is a square in $\Fqm$ then
    \begin{equation}
        \label{eq:x:divisor:1}
         (x) = (r+1)P_1 - \frac{r+1}{2}P_3 - \frac{r+1}{2}P_4,   
    \end{equation}
    while, if $-1$ is not a square in $\Fqm$ then
    \begin{equation}
        \label{eq:x:divisor:2}
         (x) = (r+1)P_1 - \frac{r+1}{2}P_3.
    \end{equation}
    From this and \cref{rem:valuations:L:generators}, we can compute the maximum possible order of vanishing of a function in $L$ in $\Fqm(\MM_r)$. Indeed, this is exactly the maximum of the degrees of the pole divisors of the monomials generating $L$. We have
    \begin{align}
    \label{eq:polediv:generators}
    \nonumber
        \mathrm{deg} \ (x^it^j)_\infty &= i(r+1) + jr ,\\
        \mathrm{deg} \ (x^{r-1}t^h)_\infty &= (r-1)(r+1) + hr,
    \end{align}
    for $i = 1,\ldots, r-2, \ j = 0, \ldots, r-1, \ h = 0,\ldots, r-2$, and hence 
    \begin{equation}
    \label{eq:polediv:max}
        \mathrm{max}_{f \in L\setminus\{0\}}\{\mathrm{deg} \ (f)_\infty\} = \mathrm{deg} \ (x^{r-1}t^{r-2})_\infty = 2r^2 - 2r - 1.
    \end{equation}
\end{remark}

By \cref{rem:valuations:L:generators} and \cref{rem:poledivs}, the following lower bound for the minimum distance $d$ of a code $\CC(B,L)$ holds.

\begin{proposition}
\label[proposition]{prop:bound:d}
    Let $\CC(B,L)$ be an $\Fqm$-linear code of length $n=b(r+1)^2$ as in \cref{def:code}, where $b\in \mathbb{Z}_{\geq 1}$. The following lower bound for the minimum distance $d$ of $\CC(B,L)$ holds:
        \begin{equation}
        \label{eq:bound:d}
            d \geq r^2(b-2) + r(2b+2) + b + 3.
        \end{equation}
\end{proposition}
\begin{proof}
    By construction, the minimum weight $\mathrm{w}_H(\mathfrak{c}_f)$ of the codeword $\mathfrak{c}_f \in \CC(B,L)$, obtained by evaluating $f \in L\setminus\{0\}$ at the points in $S_B$, is exactly $n - \#\{\mathrm{zeros \ of} \ f \ \mathrm{in} \ S_B\}$. The number of zeros of $f$ contained in $S_B$ is at most the total possible number of zeros of $f$ on $\MM_r$, which in turn is at most the degree of the pole divisor $(f)_\infty$ of $f$ in $\Fqm(\MM_r)$, that is,
    \begin{equation*}
       \mathrm{w}_H(\mathfrak{c}_f) = n - \#\{\mathrm{zeros \ of} \ f \ \mathrm{in} \ S_B\} \geq n - \#\{\mathrm{zeros \ of} \ f \ \mathrm{on} \ \MM_r\} \geq n - \mathrm{deg} \ (f)_\infty.
    \end{equation*}
By \cref{eq:polediv:max} in \cref{rem:poledivs}, 
    \begin{equation*}
        \mathrm{max}_{f \in L\setminus\{0\}}\{\mathrm{deg} \ (f)_\infty\} =  2r^2 - 2r - 1,
    \end{equation*}
    and, by \cref{eq:minval:P1} in \cref{rem:valuations:L:generators},
     \begin{equation*}
        \mathrm{min}_{f \in L\setminus\{0\}}\{v_{P_1}(f)\} = 2,
    \end{equation*}
 hence, for any $f \in L\setminus\{0\}$ the associated codeword $\mathfrak{c}_f$ has weight
    \begin{equation*}
        \mathrm{w}_H(\mathfrak{c}_f) \geq n - \mathrm{deg} \ (f)_\infty \geq n - (2r^2 - 2r - 1 - 2)= n - (2r^2 - 2r - 3) = r^2(b-2) + r(2b+2) + b + 3.
    \end{equation*}
    Note that the second inequality follows from \cref{eq:polediv:max} and \cref{eq:minval:P1}, since $P_1$ is a zero of all the functions in $L$ and is not contained in $S_B$. Hence, the bound in \cref{eq:bound:d} on the minimum distance follows.
\end{proof}

Observe that, if $b\geq 2$, the bound obtained in \cref{prop:bound:d} is $d\geq r^2(b-2) + r(2b+2) + b + 3 \geq 6r + 5 > 0$ for any possible value of $r$. On the other hand, if $b=1$, the bound in \cref{eq:bound:d} is $d\geq -r^2 + 4r + 4$, which is positive only for $r=3$. However, if $b=1$, we can actually compute the value of the minimum distance $d$ of $\CC(B,L)$ for any $r\in \mathbb{Z}_{\geq 3}$ odd (i.e., as in the assumptions of \cref{def:code}).

\begin{proposition}
    \label[proposition]{prop:mindist:beq1}
    Let $r\in \mathbb{Z}_{\geq 3}$ odd and let $\CC(B,L)$ be a code over $\Fqm$ as in \cref{def:code}, with locality $r$. If $b=1$, then the minimum distance $d$ is $d=n - (r^2 + 2r - 7) = (r+1)^2 - (r^2 + 2r - 7) = 8$.
\end{proposition}
\begin{proof}
    When $b=1$, the set of evaluation points of a code $\CC(B,L)$ is simply one set of \emph{nice} points as in \cref{eq:Sbar:def}, namely
    \begin{align*}
    S_{\bar{t}}:=\{[\bar{x}: \bar{x}^{\frac{r+1}{2}} + 1 : 1 ; \zeta^j\bar{t} : 1 ] \mid j=0,\ldots, r, \ \PP_{\bar{t}}(\bar{x})=0\}. 
    \end{align*}
    Observe that, by definition of $S_{\bar{t}}$, there are exactly $r+1$ distinct values $\bar{x}_1, \ldots, \bar{x}_{r+1}$ attained by the $x$-coordinates of points in $S_{\bar{t}}$. 
    Consider the function
    \begin{equation}
    \label{eq:function:bound}
        f_{min}:=x \cdot \prod_{i=1}^{r-1} (t - \zeta^j\bar{t}) \cdot \prod_{i=1}^{r-3} (x - \bar{x}_i) \in L \setminus\{0\},
    \end{equation}
    and note that $f_{min}$ has exactly $r^2 + 2r - 7$ distinct zeros in $S_{\bar{t}}$. Indeed, each function $t - \zeta^j\bar{t}$ has exactly $r+1$ zeros in $S_{\bar{t}}$, namely the points $[\bar{x}: \bar{x}^{\frac{r+1}{2}} + 1 : 1 ; \zeta^j\bar{t} : 1 ]$ with $i=1,\ldots,r+1$, and similarly each function $x - \bar{x}_i$ vanishes on $S_{\bar{t}}$ only at the $r+1$ points $[\bar{x}_i: \bar{x}_i^{\frac{r+1}{2}} + 1 : 1 ; \zeta^j\bar{t} : 1 ]$, with $j=1,\ldots, r+1$. On the other hand, by \cref{rem:poledivs}, the function $x$ has only one zero, namely $P_1$, which is not a point of $S_{\bar{t}}$. Therefore, each of the $r-1$ factors in $\prod_{i=1}^{r-1} (t - \zeta^j\bar{t})$ has $r+1$ distinct zeros in $S_{\bar{t}}$, and so does each of the $r-3$ factors in $\prod_{i=1}^{r-3} (x - \bar{x}_i)$. However, for each factor $x - \bar{x}_i$, $r-1$ zeros have already been counted as zeros of the $r-1$ factors in $\prod_{i=1}^{r-1} (t - \zeta^j\bar{t})$, hence we have that the function $f_{min}$ has in total $(r-1)(r+1) + (r-3)((r+1) - (r-1)) = r^2 + 2r - 7$ distinct zeros in $S_{\bar{t}}$, as claimed. Hence the codeword $\mathfrak{c}_{f_{min}}$ associated to $f_{min}$ has weight exactly
    \begin{equation}
        \label{eq:minweight:beq1}
        \mathrm{w}_H(\mathfrak{c}_{f_{min}}) = n - (r^2 + 2r - 7) = (r+1)^2 - (r^2 + 2r - 7) = 8.
    \end{equation}
    We claim that $\mathrm{w}_H(\mathfrak{c}_{f_{min}})$ is actually the minimum weight of $\CC(B,L)$, and we show this by proving that there is no function in $L\setminus\{0\}$ with more than $r^2+2r-7$ distinct zeros in $S_{\bar{t}}$. 

    Note that a function $f(x,t)=\sum_{m=1}^{r-1}a_m(t)x^m\in L \setminus\{0\}$ which, for a fixed $\bar{t}$, has $r+1$ distinct solutions as a polynomial $f(x,\bar{t})$ in $x$, must be such that $(t-\bar{t}) \mid a_m(t)$ for all $m=1,\ldots, r-1$, since the degree of $f(x,\bar{t})$ in $x$ is at most $r-1$. Hence, such an $f$ can be written as $f=c \cdot (t-\bar{t})\cdot \bar{f}(x,t)$, where $c\in \Fqm^*$, $\bar{f}\in \Fqm[x,t]$, $\mathrm{deg}_{t} \ \bar{f}\leq r-2$. Since we wish to maximize the number of points in $S_{\bar{t}}$ on which $f(x,t)=\sum_{m=1}^{r-1}a_m(t)x^m\in L \setminus\{0\}$ vanishes, we then need to have that $f$ vanishes for exactly $r-1$ distinct values of $\bar{t}$ (see \cref{eq:function:bound}), which, iterating the reasoning above, means that $f=c \cdot \prod_{i=1}^{r-1} (t - \zeta^j\bar{t})\cdot \tilde{f}(x)$. 
    
    Observe now that, symmetrically, $f=c \cdot \prod_{i=1}^{r-1} (t - \zeta^j\bar{t})\cdot \tilde{f}(x)=\sum_{m=0}^{r-1}d_m(x)t^m$, for $d_m(x)\in \Fqm[x]$ with $\mathrm{deg}\ d_m(x)\leq r-1$. Then, if $f$ has one more vanishing point in $S_{\bar{t}}$ with a certain $x$-coordinate $\bar{x}$ (i.e., if $f$ vanishes at $r$ distinct points with such an $x$-coordinate), it has to vanish on the whole subset of points of $S_{\bar{t}}$ with $x$-coordinate $\bar{x}$. Indeed, $f(\bar{x},t)$ is a polynomial in $t$ of degree $r-1$ with $r$ distinct roots, which implies that $f(\bar{x},t)$ has to vanish identically on all the points of $S_{\bar{t}}$ with $x$-coordinate equal to $\bar{x}$. Therefore, if $f$ vanishes at more than $r-1$ points with a fixed $x$-coordinate, then it vanishes at all the $r+1$ points of $S_{\bar{t}}$ with the same $x$-coordinate and, as above, we have that $(x-\bar{x}) \mid d_m(x)$ for all $m=0,\ldots, r-1$. Therefore, for $f\in L\setminus \{0\}$ to vanish on the largest possible number of points of $S_{\bar{t}}$, we necessarily have that
    \begin{equation}
    \label{eq:minimal:functions}
        f=c\cdot x \cdot \prod_{i=1}^{r-1} (t - \zeta^j\bar{t})\cdot \prod_{j=1}^{r-3} (x - \bar{x}_j),
    \end{equation}
    where $c\in \Fqm^*$ and $j$ can be at most $r-3$ because otherwise the monomial $x^{r-1}t^{r-1}\not \in L$ would appear in the expression of $f$. This hence shows that any nonzero function in $L$ with the largest possible number of vanishing points in $S_{\bar{t}}$ is exactly of the form in \cref{eq:minimal:functions}. Therefore, there is no function in $L\setminus\{0\}$ with more than $r^2+2r-7$ distinct zeros in $S_{\bar{t}}$ and the minimum distance $d$ of $\CC(B,L)$ is exactly $d=\mathrm{w}_H(\mathfrak{c}_{f_{min}}) = n - (r^2 + 2r - 7) = (r+1)^2 - (r^2 + 2r - 7) = 8$.
\end{proof}

We have thus determined the minimum distance $d$ for a code $\CC(B,L)$ with $b=1$. In the following subsection, we conclude the discussion on the codes $\CC(B,L)$ by showing that, when $b \geq 2$, the bound obtained in \cref{eq:bound:d} is sharp for $r=3$, since there are examples where the bound is attained (see \cref{table:codes}).

\begin{remark}
    \label[remark]{rem:asymptotics}
    It is not clear in which sense to study asymptotics of the family of codes $\CC(B,L)$, since when the length $n$ tends to infinity also $r$ needs to tend to infinity, which means that the field of definition $\Fqm$ cannot be fixed.
\end{remark}

\subsection{Explicit examples with \texorpdfstring{$r=3$}{r=3}}

In this subsection, we consider several examples of codes $\CC(B,L)$ as in \cref{def:code} with $r=3$. These examples show that, when $b\geq 2$, the bound on $d$ obtained in \cref{eq:bound:d} is sharp. Moreover, in all these examples, when $b\geq 2$ and the bound from \cref{eq:bound:d} is not sharp, the difference between the bound and the actual minimum distance of $\CC(B,L)$ is $1$. 

For any suitable choice of $q$, $m$ and $b$ as in \cref{def:code}, we denote by $\mathfrak{w}\in \Fqm$ a generator of the multiplicative group $\Fqm^*$, i.e. $\langle\mathfrak{w}\rangle = \Fqm^*$, and by $\zeta \in \Fqm$ a primitive $r+1$-th root of unity. Each row of \cref{table:codes} contains the parameters of a code $\CC(B,L)$ with $r=3$ defined over $\Fqm$. The first and second columns contain, respectively, the values of $q$ and $m$. The third column contains the value of $b$, that is, the number of chosen subsets of $\mathcal{G}_m$ (see \cref{eq:Gmdef:2}), and the fourth column contains the explicit expression of the set $B$ (see \cref{def:evalpoints}), showing which \emph{nice} elements of $\Fqm$ have been chosen for the construction of the set of evaluation points $S_B$ of the code $\CC(B,L)$. Note that, for $r$, $q$ and $m$ fixed, the code $\CC(B,L)$ is uniquely determined by the choice of the set of \emph{nice} elements $B$, since the $\Fqm$-vector space $L$ depends only on $r$. This explains why, in the table, it is enough to specify $q$, $m$ and $B$ to specify a certain code $\CC(B,L)$. The fifth column of the table contains the length $n$ of the code, which is $n=b(r+1)^2=b\cdot 4^2$. Observe that, as shown is \cref{lemma:genconstruction:dim}, the dimension $k$ of all the listed codes $\CC(B,L)$ is $k=r(r-1) - 1 =5$. The sixth column of \cref{table:codes} is left empty when $b=1$, since the value of the minimum distance $d$ of $\CC(B,L)$ is computed in \cref{prop:mindist:beq1} in this case. On the other hand, when $b\geq 2$, the sixth column contains the estimate $\Delta$ for the minimum distance $d$ of $\CC(B,L)$ obtained from our bound \cref{eq:bound:d}. More precisely,
\begin{itemize}
    \item if $b=1$, by \cref{prop:mindist:beq1}, the distance $d$ of any code $\CC(B,L)$ is exactly $d=8$, for any admissible $r$ as in \cref{def:code}. In this case, we leave the value corresponding to $\Delta$ empty, as we already know the exact value of $d$.
    \item If $b \geq 2$, we set instead $\Delta:=n - (2r^2 - 2r - 1 - 2)$, that is the bound obtained in \cref{eq:bound:d}. Note that $\Delta = n - 9$ for $r=3$.
\end{itemize}
Finally, the seventh column of the table contains the actual value $d$ of the minimum distance of $\CC(B,L)$, and the eighth column contains the Singleton-type upper bound $\Delta_{upper}:=16\cdot b - 5$ on the minimum distance $d$, obtained from \cref{eq:tbf:upper:explicit} with $r=3$.

For $q^m=7^2, 9^2, 11^2, 13^2$, \cref{table:codes} contains data for all the possible codes $\CC(B,L)$ with $r=3$ that can be obtained with the construction in \cref{def:code}, that is, all the possible subsets $B$ of \emph{nice} elements of $\Fqm$ are considered. For $q^m=5^4$, it can be verified that $\mathcal{G}_m$ is partitioned into exactly $8$ subsets $\mathcal{G}_m^{(\ell)}$ as in \cref{eq:Gmdef:2}, which give rise to $2^8-1=255$ possible codes $\CC(B,L)$ as in \cref{def:code}, so we choose to include in the table only a few examples. Note that in the examples included with $q^m=5^4$ and $b\geq 5$ we find that the bound in \cref{eq:bound:d} is sharp.  

In all the examples in \cref{table:codes} with $b=1$, we find $\Delta=d$ as predicted by \cref{prop:bound:d}. Moreover, the examples in the rows in bold show that the bound in \cref{eq:bound:d} is sharp, since in these cases we find $\Delta=d$. Furthermore, observe that in all the examples in the table we find that $d-\Delta \leq 1$.

\begin{table}[ht]
  \centering
  \begin{NiceTabular}{c c c c c c c c}
    \hline
    \RowStyle[bold]{}
    $\quad$ & $\quad$ & $\quad$ & $r= 3$ & $\quad$ & $\quad$ & $\quad$\\
    \hline
    $q$ & $m$ & $b$ & $B$ & $n$ & $\Delta$ & $d$ & $\Delta_{upper}$\\
    \hline
    \noalign{\vskip 0.5mm} 
    $7^2$ & $1$ & $1$ & $\{\zeta^i \mathfrak{w}^{46} \mid \ i=1,\ldots, r+1\}$ & $16$ & $-$ & $8$ & $11$\\
    $7^2$ & $1$ & $1$ & $\{\zeta^i \mathfrak{w}^6 \mid \ i=1,\ldots, r+1\}$ & $16$ & $-$ & $8$ & $11$\\
    $7^2$ & $1$ & $2$ & $\{\zeta^i \mathfrak{w}^{46}, \zeta^i \mathfrak{w}^{6} \mid \ i=1,\ldots, r+1\}$ & $32$ & $23$ & $24$ & $27$\\
    $9$ & $2$ & $1$ & $\{\zeta^i \mid \ i=1,\ldots, r+1\}$ & $16$ & $-$ & $8$ & $11$\\
    $11^2$ & $1$ & $1$ & $\{\zeta^i \mid \ i=1,\ldots, r+1\}$ & $16$ & $-$ & $8$ & $11$\\
    $11^2$ & $1$ & $1$ & $\{\zeta^i \mathfrak{w}^{15} \mid \ i=1,\ldots, r+1\}$ & $16$ & $-$ & $8$ & $11$\\
    $11^2$ & $1$ & $1$ & $\{\zeta^i 4 \mid \ i=1,\ldots, r+1\}$ & $16$ & $-$ & $8$ & $11$\\
    $11^2$ & $1$ & $2$ & $\{\zeta^i, \zeta^i \mathfrak{w}^{15} \mid \ i=1,\ldots, r+1\}$ & $32$ & $23$ & $24$ & $27$\\
    $11^2$ & $1$ & $2$ & $\{\zeta^i, \zeta^i 4 \mid \ i=1,\ldots, r+1\}$ & $32$ & $23$ & $24$ & $27$\\
    $11^2$ & $1$ & $2$ & $\{\zeta^i \mathfrak{w}^{15}, \zeta^i 4 \mid \ i=1,\ldots, r+1\}$ & $32$ & $23$ & $24$ & $27$\\
    $11^2$ & $1$ & $3$ & $\{\zeta^i, \zeta^i \mathfrak{w}^{15}, \zeta^i 4 \mid \ i=1,\ldots, r+1\}$ & $48$ & $39$ & $40$ & $43$\\
    $13$ & $2$ & $1$ & $\{\zeta^i \mathfrak{w} \mid \ i=1,\ldots, r+1\}$ & $16$ & $-$ & $8$ & $11$\\
    $13$ & $2$ & $1$ & $\{\zeta^i \mathfrak{w}^{13} \mid \ i=1,\ldots, r+1\}$ & $16$ & $-$ & $8$ & $11$\\
    $13$ & $2$ & $1$ & $\{\zeta^i \mathfrak{w}^{21} \mid \ i=1,\ldots, r+1\}$ & $16$ & $-$ & $8$ & $11$\\
    $13$ & $2$ & $1$ & $\{\zeta^i 4 \mid \ i=1,\ldots, r+1\}$ & $16$ & $-$ & $8$ & $11$\\
    $13$ & $2$ & $2$ & $\{\zeta^i \mathfrak{w}, \zeta^i \mathfrak{w}^{13} \mid \ i=1,\ldots, r+1\}$ & $32$ & $23$ & $24$ & $27$\\
    $13$ & $2$ & $2$ & $\{\zeta^i \mathfrak{w}, \zeta^i \mathfrak{w}^{21} \mid \ i=1,\ldots, r+1\}$ & $32$ & $23$ & $24$ & $27$\\
    $13$ & $2$ & $2$ & $\{\zeta^i \mathfrak{w}, \zeta^i 4 \mid \ i=1,\ldots, r+1\}$ & $32$ & $23$ & $24$ & $27$\\
    $13$ & $2$ & $2$ & $\{\zeta^i \mathfrak{w}^{13}, \zeta^i \mathfrak{w}^{21} \mid \ i=1,\ldots, r+1\}$ & $32$ & $23$ & $24$ & $27$\\
    $13$ & $2$ & $2$ & $\{\zeta^i \mathfrak{w}^{13}, \zeta^i 4 \mid \ i=1,\ldots, r+1\}$ & $32$ & $23$ & $24$ & $27$\\
    $13$ & $2$ & $2$ & $\{\zeta^i \mathfrak{w}^{21}, \zeta^i 4 \mid \ i=1,\ldots, r+1\}$ & $32$ & $23$ & $24$ & $27$\\
    $13$ & $2$ & $3$ & $\{\zeta^i \mathfrak{w}, \zeta^i \mathfrak{w}^{13}, \zeta^i \mathfrak{w}^{21} \mid \ i=1,\ldots, r+1\}$ & $48$ & $39$ & $40$ & $43$\\
    $13$ & $2$ & $3$ & $\{\zeta^i \mathfrak{w}, \zeta^i \mathfrak{w}^{13}, \zeta^i 4 \mid \ i=1,\ldots, r+1\}$ & $48$ & $39$ & $40$ & $43$\\
    $13$ & $2$ & $3$ & $\{\zeta^i \mathfrak{w}, \zeta^i \mathfrak{w}^{21}, \zeta^i 4 \mid \ i=1,\ldots, r+1\}$ & $48$ & $39$ & $40$ & $43$\\
    $13$ & $2$ & $3$ & $\{ \zeta^i \mathfrak{w}^{13}, \zeta^i \mathfrak{w}^{21}, \zeta^i 4 \mid \ i=1,\ldots, r+1\}$ & $48$ & $39$ & $40$ & $43$\\    
    ${\bf 13}$ & ${\bf 2}$ & ${\bf4}$ & ${\bf\{\zeta^i \mathfrak{w}, \zeta^i \mathfrak{w}^{13}, \zeta^i \mathfrak{w}^{21}, \zeta^i 4 \mid \ i=1,\ldots, r+1\}}$ & ${\bf64}$ & ${\bf 55}$ & ${\bf 55}$ & $59$\\
    $5$ & $4$ & $3$ & $\{\zeta^i, \zeta^i \mathfrak{w}^{9}, \zeta^i \mathfrak{w}^{13} \mid \ i=1,\ldots, r+1\}$ & $48$ & $39$ & $40$ & $43$\\
    $5$ & $4$ & $4$ & $\{\zeta^i, \zeta^i \mathfrak{w}^{9}, \zeta^i \mathfrak{w}^{13}, \zeta^i \mathfrak{w}^{33} \mid \ i=1,\ldots, r+1\}$ & $64$ & $55$ & $56$ & $59$\\
    ${\bf 5}$ & ${\bf 4}$ & ${\bf 5}$ & ${\bf \{\zeta^i, \zeta^i \mathfrak{w}^{9}, \zeta^i \mathfrak{w}^{13}, \zeta^i \mathfrak{w}^{33}, \zeta^i \mathfrak{w}^{39} \mid \ i=1,\ldots, r+1\}}$ & ${\bf80}$ & ${\bf 71}$ & ${\bf 71}$ & $75$\\
    ${\bf 5}$ & ${\bf 4}$ & ${\bf 6}$ & ${\bf \{\zeta^i, \zeta^i \mathfrak{w}^{9}, \zeta^i \mathfrak{w}^{13}, \zeta^i \mathfrak{w}^{33}, \zeta^i \mathfrak{w}^{39}, \zeta^i \mathfrak{w}^{45} \mid \ i=1,\ldots, r+1\}}$ & ${\bf96}$ & ${\bf 87}$ & ${\bf 87}$ & $91$\\
    ${\bf 5}$ & ${\bf 4}$ & ${\bf 7}$ & ${\bf \{\zeta^i, \zeta^i \mathfrak{w}^{9}, \zeta^i \mathfrak{w}^{13}, \zeta^i \mathfrak{w}^{33}, \zeta^i \mathfrak{w}^{39}, \zeta^i \mathfrak{w}^{45}, \zeta^i \mathfrak{w}^{65} \mid \ i=1,\ldots, r+1\}}$ & ${\bf112}$ & ${\bf 103}$ & ${\bf 103}$ & $107$\\
    \hline
  \end{NiceTabular}
  \caption{\label{table:codes} Each row of the table contains the value $\Delta$ and the actual minimum distance $d$ of a code $\CC(B,L)$ over $\Fqm$ with $r=3$. When $b=1$, the value of $\Delta$ is left empty, since we know from \cref{prop:mindist:beq1} that $d=8$. When $b\geq 2$, $\Delta$ is instead the lower bound for $d$ determined in \cref{eq:bound:d}, namely $\Delta:=n - (2r^2 - 2r - 1 - 2) = n - 9$. The value $\Delta_{upper}:=16\cdot b - 5$ is the Singleton-type upper bound on $d$ from \cref{eq:tbf:upper:explicit} with $r=3$.}
\end{table}

\section{Codes \texorpdfstring{$\CC(B,L)$}{C(B,L)} with \texorpdfstring{$r=3$}{r=3}: a reinterpretation with genus 1 fibrations}
\label{sec:genus1:fibrations}

In this section, we expand on \cref{rem:geometric:intuition} and rephrase the construction proposed in \cref{def:code} for $r=3$ in a more geometric perspective. This explains in which sense our construction can be viewed as a generalization of the results on LRCs from elliptic surfaces contained in \cite[Section VI]{SVAV21}.

Let $m$ be an integer satisfying \cref{eq:m:condition} with $r=3$, for some $q$, and let $\mathcal{E}_3$ over $\Fqm$ be as in \cref{eq:Erdef}. Assume moreover that $q$ is not an even power of a prime $p\equiv 3 \mod 4$. Observe that $\mathcal{E}_3$ admits two genus 1 fibrations, given by the projections $\Pi_t$ and $\Pi_x$ to the $t$-line and to the $x$-line, respectively. With a slight abuse, we refer here to the notations introduced in \cref{rem:geometric:intuition}. The curve $\mathcal{M}:=\mathcal{M}_3$ is a multisection of degree $r+1=4$ for both fibrations. Hence, in the light of \cite[Section VI]{SVAV21}, constructing a code on the points of $\mathcal{M}$ whose space of functions is simultaneously a subspace of the Riemann-Roch spaces of $4O_t$ and of $4O_x$, where $O_t$ (resp. $O_x$) is the zero section of the projection to the $t$-line (resp. the projection to the $x$-line), naturally yields a locally recoverable code with availability $2$. Observe that \cref{def:code} provides indeed an example of such construction. In what follows, we present a different take on the construction of \cref{sec:code:construction}, where we obtain the same codes as in \cref{def:code} but are able to devise a recovery strategy different than the one discussed in \cref{lemma:genconstruction:recovery}.

We start by recalling  \cref{def:code} for the case $r=3$.

\begin{definition}
    \label[definition]{def:code:req3}
    Let $q \equiv 1\pmod{4}$ and let $\E_3$ as in \cref{eq:Erdef} and $\MM_3$ as in \cref{eq:Mrdef} be defined over $\Fqm$, where $m$ is a positive integer such that $\mathcal{G}_m \neq \emptyset$.
    Choose subsets $\mathcal{G}_m^{(1)},\ldots, \mathcal{G}_m^{(b)}$ of $\mathcal{G}_m$, where $|\mathcal{G}_m|=4M$ and $b\in \mathbb{Z}$, $1\leq b \leq M$. Let $B:=\cup_{\ell=1}^b \mathcal{G}_m^{(\ell)}$ and $S_B$ as in \cref{def:evalpoints}, with $r=3$.

    Furthermore, consider the $\Fqm$-vector space
    \begin{align}
    \label{eq:L:definition:req3}
    L:= \langle xt^j, x^2t^h\mid j=0,1,2, \ h=0, 1 \rangle,
    \end{align}
    and the linear evaluation map
    \begin{align*}
        \mathrm{ev}_{S_B}: \ &L \longrightarrow \left(\Fqm\right)^n\\
        &f \longmapsto \left(f(P_{1,1}^{(1)}), \ldots, f(P_{4,4}^{(b)})\right). 
    \end{align*}
    For more compact notations, we also write $\mathrm{ev}_{S_B}(f)=\left(f(P_{i,j}^{(\ell)})\right)_{i,j=1,\ldots,4}^{\ell=1,\ldots,b}$.
    We define an AG code $\mathfrak{C}(B,L)$ of length $n:=|S_B|=b\cdot 4^2$ on $\E_3$ as the image of $\mathrm{ev}_{S_B}$, that is,
\begin{equation}
\label{eq:codedef:req3}
    \mathfrak{C}(B,L):=\left\{ \left(f(P_{i,j}^{(\ell)})\right)_{i,j=1,\ldots,4}^{\ell=1,\ldots,b} \ \biggm| \ f\in L\right\}.
\end{equation}
\end{definition}

Observe that the functions in $L$ restrict both on vertical fibers $\Pi_t^{-1}(\bar{t})$ and on horizontal fibers $\Pi_x^{-1}(\bar{x})$ as rational functions with a pole divisor of degree at most $r+1$, since $\mathrm{deg} \ (x)_\infty = 2$ (resp. $\mathrm{deg} \ (t)_\infty = 2$) on any $\Pi_t^{-1}(\bar{t})$ (resp. on $\Pi_x^{-1}(\bar{x})$). This allows us to propose an alternative recovery strategy for the codes $\CC(B,L)$ with $r=3$, which is discussed in the following lemma.

\begin{lemma}
\label[lemma]{lemma:recovery:req3}
    A code $\mathfrak{C}(B,L)$ as in \cref{def:code} has locality $r=3$ and availability $2$. More precisely, for any $f\in L\setminus\{0\}$, each symbol $f(P_{\bar{i},\bar{j}}^{(\bar{\ell})})$ of the codeword $\mathfrak{c} := \left(f(P_{i,j}^{(\ell)})\right)_{i,j=1,\ldots,4}^{\ell=1,\ldots,b}$ has two distinct recovery sets, namely
    \begin{align*}
        f(H_{\bar{i}}^{(\bar{\ell})})&:=\left\{f(P_{\bar{i},j}^{(\bar{\ell})}) \mid j=1,\ldots,4 \ \mbox{and} \ j\neq \bar{j}\right\},\\
        f(V_{\bar{j}}^{(\bar{\ell})})&:=\left\{f(P_{i,\bar{j}}^{(\bar{\ell})}) \mid i=1,\ldots,4 \ \mbox{and} \  i\neq \bar{i}\right\}.
    \end{align*}
    We refer to $f(H_{\bar{i}}^{(\bar{\ell})})$ (resp. $f(V_{\bar{j}}^{(\bar{\ell})})$) as a \emph{horizontal} (resp. \emph{vertical}) recovery set for $f(P_{\bar{i},\bar{j}}^{(\bar{\ell})})$ (see \cref{eq:horizontalset} and \cref{eq:verticalset}).
\end{lemma}
\begin{proof}
With notations as in \cref{lemma:recovery:req3}, consider $f=\sum_{s=1}^{2}a_s(t)x^s \in L\setminus\{0\}$, where $a_s(t)\in \Fqm[t]$ is a polynomial of degree $\mathrm{deg} \ a_s(t)\leq 2$ for all $s=0,1,2$, and the codeword $\mathfrak{c}$ obtained by evaluating $f$ at the points in $S$, that is,  
\begin{equation*}
    \mathfrak{c} := \left(f(P_{i,j}^{(\ell)})\right)_{i,j=1,\ldots,4}^{\ell=1,\ldots,b}.
\end{equation*}
If $\mathfrak{c}$ has a missing symbol, which without loss of generality we can assume to be $f(P_{1,1}^{(1)})$, we proceed as follows. 

\begin{enumerate}
    \item \textbf{Recovering from a vertical recovery set.}
Suppose first that we wish to recover $f(P_{1,1}^{(1)})$ using the vertical recovery set $f(V_1^{(1)})$.
We start by observing that $P_{1,1}^{(1)}$ and the points in the vertical set $V_1^{(1)}$, which lie in the intersection of the vertical fiber $\Pi_t^{-1}(\bar{t}_1^{(1)})$ and the multisection $\mathcal{M}$, are in particular the affine points of intersection of the elliptic curve $\Pi_t^{-1}(\bar{t}_1^{(1)})$ with the conic defined by $y=x^2+1$ in $\A_{(x,y)}^2$. In fact, denoting by $O$ the point at infinity of $\Pi_t^{-1}(\bar{t}_1^{(1)})$, the intersection divisor of the curves $\Pi_t^{-1}(\bar{t}_1^{(1)})$ and $y=x^2+1$ in $\A_{(x,y)}^2$ can be written as $P_{1,1}^{(1)} + P_{2,1}^{(1)} + P_{3,1}^{(1)} + P_{4,1}^{(1)} + 2O$. Since $\Pi_t^{-1}(\bar{t}_1^{(1)})$ is in Weierstrass form and $O$ is an inflection point, we have that $P_{1,1}^{(1)} + P_{2,1}^{(1)} + P_{3,1}^{(1)} + P_{4,1}^{(1)} + 2O = O$ in the group law of $\Pi_t^{-1}(\bar{t}_1^{(1)})$, which implies in particular that $P_{1,1}^{(1)} + P_{2,1}^{(1)} + P_{3,1}^{(1)} + P_{4,1}^{(1)} = O$. 

Observe now that the elliptic curve $\mathcal{E}_3: y^2=x^3-x^2(t^4+1)+xt^4$ over $\Fqm(t)$ yields an elliptic fibration on a K3 surface, with generic fiber that has finite Mordell--Weil group. Indeed, the fibers of the projection $\Pi_t$ to the $t$-line are clearly elliptic curves, and the induced fibration has a natural section, namely $([0:1:0], t)$. Therefore, the projection is indeed an elliptic fibration. Moreover, the underlying surface is K3 since the degree of the coefficients satisfy $v(a_i(t))\leq 2i$ and $v(a_2(t))=4$ (see the discussion in \cite[Subsect.4.10]{ShiodaSchuett}). 
The fact that the Mordell--Weil group is finite is a direct application of the Shioda-Tate formula (\cite[Cor. 6.13]{ShiodaSchuett}), as we explain in what follows. 

After an inspection of the discriminant, we obtain that the singular fibers are of types $2I_8, 4I_2$. Let $\rho$ be the Picard number of $\E_3$. Then, since $\E_3$ is K3, its Picard number satisfies  $1\leq \rho \leq 20$. Let $\mathfrak{r}$ be the the rank of the Mordell--Weil group of the generic fiber of $\Pi_t$. Then, the Shioda-Tate formula yields, under the assumption that $q$ is not an even power of a prime $p\equiv 3 \mod 4$, 
\[
\rho= \mathfrak{r} +2 + 2\cdot7 + 4\cdot1 = r+ 20.
\]
Since $\rho \leq 20$ and $\mathfrak{r}\geq 0$, we conclude that $\rho=20$ and therefore $\mathfrak{r}=0$ as claimed.

We can hence apply \cite[Lemma VI.2]{SVAV21}, which ensures that the evaluation points in the vertical set $V_1^{(1)}$ sum to $O$ in the group law of $\Pi_t^{-1}(\bar{t}_1^{(1)})$. To conclude it is then enough to note that, since no point in $S_B$ coincides with $O$ (by \cref{def:code:req3} and the fact that $\Pi_t^{-1}(\bar{t}_1^{(1)})$ is in Weierstrass form), we can apply \cite[Lemma VI.1]{SVAV21}. Indeed, the assumption that no point in $V_1^{(1)}$ is $O$, while the sum of all the points in $V_1^{(1)}$ is equal to $O$, is actually enough for the argument used in the proof of \cite[Lemma VI.1]{SVAV21}. This shows that the recoverability property holds for the vertical recovery sets.

\item 
\textbf{Recovering from a horizontal recovery set.} 
Suppose now that we wish to recover $f(P_{1,1}^{(1)})$ using the horizontal recovery set $f(H_1^{(1)})$. Observe first that, differently from $\Pi_t$, the fibration $\Pi_x$ does not admit a section. Indeed, write $\Pi_x$ as a family of hyperelliptic curves in the following way:
\begin{equation}\label{eq: genus 1 no section}
y^2=(t^4-x)(-x^2+x).
\end{equation}
It can then be seen that the point $[0:1:0]$ only splits in the desingularisation when $-x^2+x$ is a square. Therefore, even if there are several fibers that yield elliptic curves, $\Pi_x$ does not admit a section. However, we note the following:

\begin{enumerate}
\item[i)] The fibration (\ref{eq: genus 1 no section}) is a quadratic twist by $\mathfrak{d}(x)=-x^2+x$ of the rational elliptic fibration with Mordell--Weil rank 0 and singular fibers of types $III, III^*$:
\begin{equation}\label{eq: RES}
y^2=t^4-x.
\end{equation}
\item[ii)] Since the twist happens above $x=0,1$, the fibration \ref{eq: genus 1 no section} has singular fibers of types $2III^*$ and $I_0^*$.
\item[iii)] Furthermore, applying the Shioda--Tate formula (\cite[Cor. 6.13]{ShiodaSchuett}) to the Jacobian fibration, we obtain that the Mordell--Weil rank of the generic fiber of the Jacobian fibration is 0, and by \cite[Thm 2.4]{Shimada} we conclude that the Mordell--Weil group is $\mathbb{Z}/2\mathbb{Z}$.
\item[iv)] Now, from iii), we know that the trace of the image of $\mathcal{M}$ in the Jacobian fibration is contained in the 2-torsion of the generic fiber. By \cite[Specialization thm.]{Silverman}, the sum of the points in the intersection of the image of $\mathcal{M}$ and the Jacobian of an arbitrary fiber $\Pi_x^{-1}(\bar{x})$ of $\Pi_x$ lies in the 
$2$-torsion points of the fiber $\Pi_x^{-1}(\bar{x})[2]$.
\end{enumerate}
By construction (see \cref{def:code:req3}), the evaluation points of $\CC(B,L)$ all lie on fibers of $\Pi_x$ that are elliptic curves, i.e., that are above points $\bar{x}$ such that $-\bar{x}^2+\bar{x}$ is a square. Since these fibers coincide with their Jacobian, we hence have that, by $\mathrm{iii)}$ and $\mathrm{iv)}$ above, the hypotheses of \cite[Lemma VI.1]{SVAV21} and \cite[Lemma VI.2]{SVAV21} are satisfied, and hence the recovery procedure described in \cite[Lemma VI.1]{SVAV21} works in our case as well. This shows that the recoverability property of $\CC(B,L)$ holds also for the horizontal recovery sets, concluding the proof.
\end{enumerate}
\end{proof}

\begin{remark}
    In the proof of \cref{lemma:recovery:req3}, the recovery procedure from a vertical recovery set could also be proved without using \cite[Lemmas VI.1 and VI.2]{SVAV21}, while still using a more geometric argument than the one in the proof of \cref{lemma:genconstruction:recovery}. Indeed, to show that $f(P_{1,1}^{(1)})$ can be uniquely recovered from the values of $f$ at the points in the vertical set $V_1^{(1)}=\{P_{2,1}^{(1)},P_{3,1}^{(1)},P_{4,1}^{(1)}\}$ one could also argue in the following way.

    Let $f_1,f_2$ be the restrictions to $\Pi_t^{-1}(\bar{t}_1^{(1)})$ of two rational functions in $L\setminus \{0\}$ such that $f_1(P_{i,1}^{(1)})=f_2(P_{i,1}^{(1)})$, for $i=2,3,4$, and $f_1(P_{1,1}^{(1)}) \neq f_2(P_{1,1}^{(1)})$, and consider the function on $\Pi_t^{-1}(\bar{t}_1^{(1)})$ given by their difference $g:=f_1-f_2$. Observe that, since $f_1,f_2\in \mathscr{L}(\Pi_t^{-1}(\bar{t}_1^{(1)}),4O)$, where $\mathscr{L}(\Pi_t^{-1}(\bar{t}_1^{(1)}),4O)$ denotes the Riemann-Roch space of the divisor $4O$ on $\Pi_t^{-1}(\bar{t}_1^{(1)})$, we also have that $g\in\mathscr{L}(\Pi_t^{-1}(\bar{t}_1^{(1)}),4O)$. More precisely, the principal divisor of $g$ on $\Pi_t^{-1}(\bar{t}_1^{(1)})$ can be written as $(g) = D + P_{2,1}^{(1)} + P_{3,1}^{(1)} + P_{4,1}^{(1)} - \mathfrak{m}O$, where $D$ is an effective divisor whose support might contain $P_{2,1}^{(1)},P_{3,1}^{(1)},P_{4,1}^{(1)}$, but does not contain $P_{1,1}^{(1)}$, and $\mathfrak{m}:=\mathrm{max}\{\mathrm{deg}(f_1)_\infty, \mathrm{deg}(f_2)_\infty\}$. Note that, by construction, $3\leq \mathfrak{m} \leq 4$, which means that $\mathfrak{m}$ can only be either $3$ or $4$. 
\begin{itemize}
    \item If $\mathfrak{m}=3$, then the support of the divisor $D$ must be empty, that is, $(g) = P_{2,1}^{(1)} + P_{3,1}^{(1)} + P_{4,1}^{(1)} - 3O$, which by Abel's Theorem (see \cite[Corollary III 3.5]{TVN}) implies that $P_{2,1}^{(1)} + P_{3,1}^{(1)} + P_{4,1}^{(1)} = O$. Since we already have that $P_{1,1}^{(1)} + P_{2,1}^{(1)} + P_{3,1}^{(1)} + P_{4,1}^{(1)} = O$, this yields $P_{1,1}^{(1)}=O$, which is a contradiction.
    \item If instead $\mathfrak{m}=4$, then the support of $D$ must contain exactly one point $P$, that is, $(g) = P + P_{2,1}^{(1)} + P_{3,1}^{(1)} + P_{4,1}^{(1)} - 4O$. By Abel's Theorem, this implies that $P + P_{2,1}^{(1)} + P_{3,1}^{(1)} + P_{4,1}^{(1)} = O$. Combining this with $P_{1,1}^{(1)} + P_{2,1}^{(1)} + P_{3,1}^{(1)} + P_{4,1}^{(1)} = O$, we hence obtain that $P=P_{1,1}^{(1)}$, which means that $P_{1,1}^{(1)}$ is a zero of $g$, implying that $f_1(P_{1,1}^{(1)})=f_2(P_{1,1}^{(1)})$ and contradicting our assumptions on $f_1$ and $f_2$.
\end{itemize}
This shows that any two functions in $L$ that coincide on the points in $V_1^{(1)}$ also have the same value on $P_{1,1}^{(1)}$, hence the recoverability property holds for the vertical recovery sets.    
\end{remark}

\section*{Acknowledgments}
The authors would like to sincerely thank Ander Arriola Corpion, for a relevant comment on the geometry of the genus 1 fibrations in \cref{sec:genus1:fibrations}, and Mauricio Frieri, for pointing out an issue in the first version of \cref{lemma:existence:of:evalpoints}.

This work was partially supported by an NWO Open Competition ENW – XL grant (project \enquote{Rational points: new dimensions} - OCENW.XL21.XL21.011). We thank all the members of the consortium for the welcoming and stimulating atmosphere.

Part of this work was done while C. Salgado was visiting the Institute for Advanced Study, in Princeton. It is a pleasure to thank the institute and its staff for the amazing work environment, and to acknowledge the support of the James D. Wolfensohn Fund and the National Science Foundation under Grant No. DMS-1926686. 

L. Vicino is supported by a DFF-International Postdoc Grant (grant ID: 10.46540/5246-00028B) from Independent Research Fund Denmark.

\end{document}